\documentclass[11pt,reqno,a4paper]{amsart}

\usepackage[T1]{fontenc}
\usepackage[utf8]{inputenc}

\usepackage{libertinus}
\usepackage{microtype}

\usepackage{geometry}

\usepackage{xcolor}

\usepackage{amsmath,amssymb,amsthm,amsfonts}
\usepackage{mathtools}
\usepackage{mathrsfs}
\usepackage{graphicx}

\usepackage[
colorlinks=true,
linkcolor=blue,
citecolor=blue,
urlcolor=blue
]{hyperref}

\usepackage[nameinlink,noabbrev]{cleveref}

\usepackage{titlesec}

\titleformat{\section}
  {\normalfont\large\bfseries\centering}
  {\thesection.}
  {0.75em}
  {}

\titlespacing*{\section}
  {0pt}
  {3.2ex plus 1ex minus .2ex}
  {1.6ex plus .3ex}

\titleformat{\subsection}
  {\normalfont\normalsize\bfseries}
  {\thesubsection.}
  {0.75em}
  {}

\titlespacing*{\subsection}
  {0pt}
  {2.4ex plus .8ex minus .2ex}
  {1.0ex plus .2ex}

\titleformat{\subsubsection}
  {\normalfont\normalsize\bfseries\itshape}
  {\thesubsubsection.}
  {0.75em}
  {}

\titlespacing*{\subsubsection}
  {0pt}
  {2.0ex plus .6ex minus .2ex}
  {0.7ex plus .2ex}
\numberwithin{equation}{section}

\usepackage{aliascnt}

\newtheoremstyle{compact}
  {6pt}   % Space above
  {1 pt}   % Space below
  {\itshape}
  {}
  {\bfseries}
  {.}
  {0.5em}
  {}

\theoremstyle{compact}

\newtheorem{theorem}{Theorem}[section]

\newaliascnt{proposition}{theorem}
\newtheorem{proposition}[proposition]{Proposition}
\aliascntresetthe{proposition}

\newaliascnt{lemma}{theorem}
\newtheorem{lemma}[lemma]{Lemma}
\aliascntresetthe{lemma}

\newaliascnt{corollary}{theorem}
\newtheorem{corollary}[corollary]{Corollary}
\aliascntresetthe{corollary}

\theoremstyle{definition}
\newaliascnt{definition}{theorem}
\newtheorem{definition}[definition]{Definition}
\aliascntresetthe{definition}

\newtheorem*{question}{Question}
\newtheorem*{conjecture}{Conjecture}

\theoremstyle{remark}
\newaliascnt{remark}{theorem}
\newtheorem{remark}[remark]{Remark}
\aliascntresetthe{remark}

\usepackage{etoolbox}

\AtBeginEnvironment{proof}{\vspace{6pt}}

\usepackage{enumitem}

\setlist{
itemsep=2pt,
topsep=4pt
}

\usepackage[toc,page]{appendix}

\newcommand\subsetsim{\mathrel{%
\ooalign{\raise0.2ex\hbox{$\subset$}\cr\hidewidth\raise-0.8ex\hbox{\scalebox{0.9}{$\sim$}}\hidewidth\cr}}}

\DeclareMathOperator{\Per}{Per}

\DeclareMathOperator{\Ann}{Ann}

\usepackage{comment}

\makeatletter
\renewcommand{\l@section}[2]{%
  \par\addpenalty\@secpenalty
  \addvspace{1.0em plus 1pt}%
  \@tempdima 1.5em
  \begingroup
    \parindent \z@
    \rightskip \@pnumwidth
    \parfillskip -\@pnumwidth
    \leavevmode
    \bfseries
    \advance\leftskip\@tempdima
    \hskip -\leftskip
    #1\nobreak\hfil \nobreak\hb@xt@\@pnumwidth{\hss #2}\par
  \endgroup}
\makeatother

\usepackage{bbm}
\usepackage{amscd}
\usepackage[all,cmtip]{xy}
\usepackage{changepage}
\usepackage{calc}
\usepackage{scalerel}
\usepackage{extarrows}
\usepackage{aliascnt}

\DeclareMathSymbol{\shortminus}{\mathbin}{AMSa}{"39}

\newcommand{\ignore}[1]{}

\begin{document}

\title{Directional expansion in ergodic actions of countable groups}

\author[M. Bj\"orklund]{Michael Bj\"orklund}
\address{Department of Mathematics, Chalmers University of Technology and University of Gothenburg, Gothenburg, Sweden}
\email{micbjo@chalmers.se}

\author[A. Fish]{Alexander Fish}
\address{School of Mathematics and Statistics, University of Sydney, NSW 2006, Australia}
\email{alexander.fish@sydney.edu.au}


\begin{abstract}
We study directional expansion for probability-measure-preserving actions
of countable groups through a representation-theoretic group property,
the cyclic escape property. An infinite countable group has the cyclic
escape property if every totally ergodic unitary representation has
arbitrarily small fixed-vector projections along infinite cyclic
subgroups. This property implies directional expansivity for all totally
ergodic actions. We prove that all infinite finitely generated nilpotent
groups have the cyclic escape property, and conjecture the same for all
infinite finitely generated polycyclic groups. We also prove the cyclic
escape property for higher-rank simple lattices whose finite-dimensional
unitary representations all have finite image; in particular, for
\(\operatorname{SL}_n(\mathbb Z)\),
\(\operatorname{PSL}_n(\mathbb Z)\), and
\(\operatorname{PGL}_n(\mathbb Z)\), \(n\geq 3\). By contrast, free
groups of rank at least two do not have the cyclic escape property. The
proofs exhibit two independent mechanisms: central spectral structure in
nilpotent groups and stationary character rigidity in higher-rank
lattices.
\end{abstract}

\maketitle

\section{Introduction}

\subsection{Directional expansion}
\label{subsec:dir}

A basic problem in ergodic theory is to understand how an action behaves
after restriction to subgroups. For actions of \(\mathbb Z^d\) and
\(\mathbb R^d\), this leads to directional notions of ergodicity, weak
mixing, and mixing. The question of when an ergodic action has an
ergodic element was studied by Pugh and Shub~\cite{Pugh_Shub_1971};
directional versions for \(\mathbb Z^d\)- and \(\mathbb R^d\)-actions
were studied by Robinson, Rosenblatt and
\c{S}ahin~\cite{Robinson_Rosenblatt_Sahin_2022}.

There is also a related directional theory for infinite-measure actions
of nilpotent lattices: Danilenko studied recurrence and rigidity along
one-parameter subgroups of the ambient simply connected nilpotent Lie
group~\cite{Danilenko_2017}. Our setting is different. We work with
probability-measure-preserving actions of countable groups, and the
directions are cyclic subgroups of the acting group itself.

Let \(G\) be a countable discrete group, and let
\(G\curvearrowright (X,\mu)\) be a probability-measure-preserving
action.

\begin{definition}[Directional expansion]
\label{def:directionally-expansive-action}
The action is \emph{directionally expansive} if, for every measurable
set \(B\subset X\) with \(\mu(B)>0\) and every \(\varepsilon>0\), there
exists a cyclic subgroup \(C\leq G\) such that
\[
        \mu(CB)>1-\varepsilon,
        \qquad
        CB=\bigcup_{c\in C}cB .
\]
\end{definition}

Directional expansivity implies ergodicity. Indeed, if \(A\subset X\)
is \(G\)-invariant and has positive measure, then \(CA=A\) for every
cyclic subgroup \(C\leq G\). Applying the definition to \(B=A\) gives
\(\mu(A)>1-\varepsilon\) for every \(\varepsilon>0\), and hence
\(\mu(A)=1\).

We allow \(C\) to be finite. This is immaterial for the positive results
below, where the expanding cyclic subgroups will be infinite. For
general countable groups, it avoids excluding torsion groups by
definition.

The positive results are stated for totally ergodic actions. An action
is \emph{totally ergodic} if its restriction to every finite-index
subgroup of \(G\) is ergodic. This assumption rules out finite factor
obstructions: ordinary ergodicity alone does not suffice; see
\cite[Example~1.7]{Bjorklund_Fish_2024}. Directional expansivity is
nevertheless weaker than asking for an ergodic cyclic subaction, since
even weakly mixing \(\mathbb Z^r\)-actions need not have one; see
\cite[Example~1.6]{Bjorklund_Fish_2024}.

For free abelian groups, directional expansivity was introduced by the
authors~\cite{Bjorklund_Fish_2024} and used to study volume spectra of
simplices in large subsets of free abelian groups. It has since been
used to obtain qualitative and quantitative results for various
geometric and combinatorial spectra; see
\cite{Bjorklund_Cullman_Fish_2025,Fish_Skinner_2025,BFS26,CS26}. The
present paper studies new classes of countable groups for which total
ergodicity forces directional expansivity.

%%%%%%%%%%%%%%%%%%%%%%%%%%%%%%%%%%%%%%%%%%%%%%%%%%%%%%%%%%%%%%%%%%%%%%

\subsection{The cyclic escape property}
\label{subsec:cyclic-escape-property}

We isolate the representation-theoretic group property used to prove
directional expansivity.

Let \((V,\pi)\) be a unitary representation of a countable group \(G\).
If \(H\leq G\), write
\[
        V^H=\{v\in V:\pi(h)v=v\text{ for every }h\in H\},
\]
and let \(P_H\) denote the orthogonal projection onto \(V^H\). The
representation \((V,\pi)\) is called \emph{totally ergodic} if $V^H=\{0\}$
for every finite-index subgroup \(H\leq G\).

\begin{definition}[Cyclic escape property]
\label{def:cyclic-escape-property}
An infinite countable group \(G\) has the \emph{cyclic escape property}
if, for every totally ergodic unitary representation \((V,\pi)\) of
\(G\) and every \(v\in V\),
\[
        \inf_C \|P_Cv\|=0,
\]
where the infimum is over all infinite cyclic subgroups \(C\leq G\).
\end{definition}

The cyclic escape property is a group-level sufficient condition for
directional expansivity of totally ergodic actions.

\begin{lemma}
\label{lem:cyclic-escape-implies-directional-expansion}
If \(G\) has the cyclic escape property, then every totally ergodic
probability-measure-preserving action of \(G\) is directionally
expansive.
\end{lemma}

\begin{proof}
Let $G\curvearrowright (X,\mu)$ be a totally ergodic probability-measure-preserving action. Let
\(B\subset X\) be measurable with \(\mu(B)>0\), and put
\[
        v=1_B-\mu(B)\in L^2_0(X,\mu).
\]
The Koopman representation on \(L^2_0(X,\mu)\) is totally ergodic in the
sense above.

Let \(C\leq G\) be an infinite cyclic subgroup, and let
\(\mathcal I_C\) be the \(C\)-invariant sigma-algebra. Then
\[
        \mathbb E(1_B\mid\mathcal I_C)=\mu(B)+P_Cv.
\]
This conditional expectation is supported on \(CB\). Hence
\[
\begin{aligned}
        \mu(B)^2
        &=
        \left(
        \int_{CB}\mathbb E(1_B\mid\mathcal I_C)\,d\mu
        \right)^2  \\
        &\leq
        \mu(CB)
        \int_{CB}
        \left|\mathbb E(1_B\mid\mathcal I_C)\right|^2\,d\mu \\
        &\leq
        \mu(CB)
        \left(
        \mu(B)^2+\|P_Cv\|_2^2
        \right).
\end{aligned}
\]
Since \(G\) has the cyclic escape property, \(C\) can be chosen so that
\(\|P_Cv\|_2\) is arbitrarily small. Therefore, for every
\(\varepsilon>0\), some infinite cyclic subgroup \(C\leq G\) satisfies
\[
        \mu(CB)>1-\varepsilon .
\]
Thus the action is directionally expansive.
\end{proof}

The cyclic escape property is not restricted to finitely generated
groups. 
Indeed,
\[
        \mathbb Q=\bigcup_{n\geq 1}(1/n!)\mathbb Z .
\]
Let $C_n=(1/n!)\mathbb Z$. Then \(C_n\leq C_{n+1}\), and hence the fixed subspaces \(V^{C_n}\)
decrease to \(V^{\mathbb Q}\). An ergodic representation $(V,\pi)$ of
\(\mathbb Q\) satisfies $V^{\mathbb Q}=\{0\}$. Therefore
\[
        P_{C_n}v\longrightarrow 0
\]
for every vector \(v \in V\). This proves the cyclic escape property for
\(\mathbb Q\). Consequently, every ergodic probability-measure-preserving
action of \(\mathbb Q\) is directionally expansive.

%%%%%%%%%%%%%%%%%%%%%%%%%%%%%%%%%%%%%%%%%%%%%%%%%%%%%%%%%%%%%%%%%%%%%%

\subsection{Nilpotent groups}
\label{subsec:nilpotent-groups}

The first main result is that finitely generated nilpotent groups have
the cyclic escape property.

\begin{theorem}
\label{thm:main}
Every infinite finitely generated nilpotent group has the cyclic escape
property. Consequently, every totally ergodic
probability-measure-preserving action of such a group is directionally
expansive.
\end{theorem}

The second assertion follows from
Lemma~\ref{lem:cyclic-escape-implies-directional-expansion}. The proof
of the cyclic escape property is by induction on the nilpotency class.

In the abelian case, total ergodicity excludes atoms at torsion
characters, and Fejér averages give cyclic subgroups with small
fixed-vector projections. For the induction step, assume that \(G\) has
nilpotency class \(r\geq 2\), and let
\[
        Z=\gamma_r(G)
\]
be the last nontrivial term of the lower central series. Then \(Z\) is
finitely generated, abelian and central, while \(G/Z\) has smaller
nilpotency class.

The centrality of \(Z\) is the key point. It makes the spectral
decomposition over \(\widehat Z\) invariant under the whole group. Given
a vector \(v\), we split it into periodic and aperiodic central spectral
parts. The periodic part is reduced to the induction hypothesis through
quotients by finite-index subgroups of \(Z\). The aperiodic part is
controlled by averaging over finite-index subgroups of \(Z\). Combining
the two pieces gives an infinite cyclic subgroup \(C\leq G\) for which
$\|P_Cv\|$
is arbitrarily small.

%%%%%%%%%%%%%%%%%%%%%%%%%%%%%%%%%%%%%%%%%%%%%%%%%%%%%%%%%%%%%%%%%%%%%%

\subsection{The polycyclic conjecture}
\label{subsec:polycyclic-conjecture}

The nilpotent theorem suggests the following extension.

\begin{conjecture}
\label{conj:polycyclic-cyclic-escape}
\emph{Every infinite finitely generated polycyclic group has the cyclic escape
property.}
\end{conjecture}

The conjecture asks whether the nilpotent result is part of a broader
solvable-group phenomenon. The difficulty appears to be a uniformity
problem: different cyclic subgroups may work on individual irreducible pieces,
whereas the cyclic escape property requires one cyclic subgroup which
gives a small projection for a prescribed vector in an arbitrary
totally ergodic representation.

%%%%%%%%%%%%%%%%%%%%%%%%%%%%%%%%%%%%%%%%%%%%%%%%%%%%%%%%%%%%%%%%%%%%%%

\subsection{Higher-rank simple lattices}
\label{subsec:higher-rank-simple-lattices}

The cyclic escape property is not confined to nilpotent or solvable
groups. We also prove it for a large class of higher-rank lattices, by a
method unrelated to the nilpotent argument.

For brevity, we call a lattice in a connected noncompact simple Lie
group with finite centre and real rank at least two a
\emph{higher-rank simple lattice}.

\begin{theorem}
\label{thm:higher-rank-lattices}
Let \(L\) be a connected noncompact simple Lie group with finite centre
and real rank at least two, and let \(G<L\) be a lattice.

\begin{enumerate}
\item
If \(L\) has trivial centre, then every weakly mixing
probability-measure-preserving action of \(G\) is directionally
expansive.

\item
If every finite-dimensional unitary representation of \(G\) has finite
image, then \(G\) has the cyclic escape property. Consequently, every
totally ergodic probability-measure-preserving action of \(G\) is
directionally expansive.
\end{enumerate}
\end{theorem}

The proof is based on stationary character rigidity for higher-rank
lattices, due to Boutonnet--Houdayer. In the trivial-centre case, this
gives the weakly mixing statement. The additional finite-image
hypothesis in part (2) is used to pass from total ergodicity to weak
mixing; the finite-centre case is then obtained by passing to the
adjoint quotient and using the finite-kernel transfer lemma.

In Section~\ref{sec:higher-rank-lattices}, we verify the finite-image
hypothesis for standard arithmetic examples. In particular, we obtain
the following consequence.

\begin{corollary}
\label{cor:SL-PSL-PGLnZ}
For every \(n\geq 3\), the groups
\[
        \operatorname{SL}_n(\mathbb Z),\qquad
        \operatorname{PSL}_n(\mathbb Z),\qquad
        \operatorname{PGL}_n(\mathbb Z)
\]
have the cyclic escape property. Consequently, every totally ergodic
probability-measure-preserving action of any of these groups is
directionally expansive.
\end{corollary}

%%%%%%%%%%%%%%%%%%%%%%%%%%%%%%%%%%%%%%%%%%%%%%%%%%%%%%%%%%%%%%%%%%%%%%

\subsection{Obstructions}
\label{subsec:obstructions}

The preceding positive results do not extend to arbitrary countable
groups. We record two obstructions, showing in particular that total
ergodicity does not imply directional expansivity in general.

\begin{theorem}[Dense subgroups in compact groups with thin cyclic subgroups]
\label{thm:compact-obstruction}
Let \(\Gamma\) be a countable dense subgroup of \(SU(2)\). Then the
translation action
\[
        \Gamma\curvearrowright (SU(2),m_{SU(2)})
\]
is totally ergodic but not directionally expansive. In particular,
\(\Gamma\) does not have the cyclic escape property. Consequently, every free group of rank at least two fails to have the
cyclic escape property.
\end{theorem}

\begin{proof}
If \(H\leq \Gamma\) has finite index, then \(\overline H\) is a closed
finite-index subgroup of \(SU(2)\). Hence \(\overline H\) is open. Since
\(SU(2)\) is connected, it has no proper open subgroup, and therefore
\[
        \overline H=SU(2).
\]
Thus the restriction of the translation action to \(H\) is ergodic, and
the action is totally ergodic.

Let \(T\leq SU(2)\) be a maximal torus. Choose a conjugation-invariant
neighbourhood \(B\) of the identity such that
\[
        0<m_{SU(2)}(B)
        \qquad\text{and}\qquad
        m_{SU(2)}(TB)<1 .
\]
To see that this is possible, choose a decreasing
conjugation-invariant neighbourhood basis \(B_n\) of the identity with
\[
        \bigcap_n B_n=\{e\}.
\]
Then
\[
        TB_n\downarrow T,
\]
and hence
\[
        m_{SU(2)}(TB_n)\to m_{SU(2)}(T)=0,
\]
since \(T\) is Haar-null in \(SU(2)\).

Every element of \(SU(2)\) lies in a maximal torus, and all maximal
tori in \(SU(2)\) are conjugate. Hence, for every \(\gamma\in\Gamma\),
there is \(g\in SU(2)\) such that
\[
        \langle \gamma\rangle \leq gTg^{-1}.
\]
Since \(B\) is conjugation invariant and Haar measure is conjugation
invariant,
\[
        m_{SU(2)}(\langle\gamma\rangle B)
        \leq
        m_{SU(2)}(gTg^{-1}B)
        =
        m_{SU(2)}(TB).
\]
Taking
\[
        \varepsilon=1-m_{SU(2)}(TB)>0,
\]
no cyclic subgroup expands \(B\) to measure greater than
\(1-\varepsilon\). Thus the action is not directionally expansive.

The final assertion follows from the existence of dense free subgroups
of connected semisimple Lie groups, see for instance
\cite{Breuillard_Gelander_2003}.
\end{proof}

\begin{remark}
The same argument applies with \(SU(2)\) replaced by any compact
connected nonabelian Lie group \(K\). Indeed, every element of \(K\) is
contained in a maximal torus, all maximal tori are conjugate, and a
maximal torus is Haar-null in \(K\). Thus, if a countable group admits a
homomorphism with dense image in such a \(K\), then the corresponding
translation action on \(K\) is totally ergodic, not weakly mixing, and
not directionally expansive.
\end{remark}

\begin{theorem}[Bounded torsion quotients]
\label{thm:torsion-quotient-obstruction}
Let \(G\) be a countable discrete group admitting an infinite quotient
of bounded exponent. Then \(G\) admits a totally ergodic
probability-measure-preserving action which is not directionally
expansive. In particular, \(G\) does not have the cyclic escape property.
\end{theorem}

\begin{proof}
Let
\[
        q:G\to Q
\]
be an infinite quotient of bounded exponent, and let \(m\) bound the
orders of elements of \(Q\). Let \(Q\) act on its Bernoulli shift
\[
        (X,\mu)=(\{0,1\}^Q,(1/2,1/2)^Q),
\]
and view this as a \(G\)-action through \(q\).

The Bernoulli shift of \(Q\) is mixing, and its restriction to every
infinite subgroup of \(Q\) is ergodic. If \(H\leq G\) has finite index,
then \(q(H)\) has finite index in \(Q\), and hence is infinite. Thus the
restriction of the \(G\)-action to \(H\) is ergodic. Therefore the
\(G\)-action is totally ergodic.

Choose \(B\subset X\) with
\[
        0<\mu(B)<\frac{1}{m}.
\]
For every cyclic subgroup \(C\leq G\),
\[
        \mu(CB)
        =
        \mu(q(C)B)
        \leq
        |q(C)|\mu(B)
        \leq
        m\mu(B)
        <1 .
\]
Hence the action is not directionally expansive.
\end{proof}

By work of Ivanov--Olshanskii, in the form used by Minasyan,
Olshanskii and Sonkin~\cite{Minasyan_Olshanskii_Sonkin_2009}, every
non-elementary hyperbolic group has an infinite quotient of bounded
exponent. Hence non-elementary hyperbolic groups do not have the cyclic
escape property.

These examples show that the cyclic escape property is a genuine
restriction on the acting group. In particular, the nilpotent and
higher-rank positive results cannot be extended to all countable groups.

%%%%%%%%%%%%%%%%%%%%%%%%%%%%%%%%%%%%%%%%%%%%%%%%%%%%%%%%%%%%%%%

\subsection{Finite cyclic directions}
\label{subsec:finite-cyclic-directions}

Directional expansivity allows finite cyclic subgroups, while the cyclic
escape property uses only infinite cyclic subgroups. Thus torsion groups
fall outside the scope of the cyclic escape criterion but not outside
the original dynamical question. A natural test case is the infinite
alternating group
\[
        A_\infty=\bigcup_{n\geq 1}A_n,
\]
under the natural inclusions, which is simple and contains finite cyclic subgroups of arbitrarily large
order.

\begin{question}
Does every totally ergodic probability-measure-preserving action of
\(A_\infty\) have directional expansivity?
\end{question}

Equivalently, one may ask whether a torsion analogue of the cyclic
escape property holds for \(A_\infty\), with finite cyclic subgroups of
large order replacing infinite cyclic subgroups. Thoma's character
theorem for \(S_\infty\)~\cite{Thoma_1964} suggests a possible
representation-theoretic approach, analogous in spirit to the
character-theoretic part of the proof of
Theorem~\ref{thm:higher-rank-lattices}.

%%%%%%%%%%%%%%%%%%%%%%%%%%%%%%%%%%%%%%%%%%%%%%%%%%%%%%%%%%%%%%%%%%%%%%

\subsection{Acknowledgments}
M.B. was supported by the Swedish Research Council VR 11253322. A.F. 
would like to thank the Australian Research Council for support through the
grant DP240100472.

\subsection{Use of AI tools.}
During the preparation of this manuscript, the authors used ChatGPT for
exploratory discussion, proof checking, organization, and drafting
assistance. The Aristotle API was used to a lesser extent for checking
selected arguments. The authors reviewed all AI-generated output and take
full responsibility for the mathematical content and final text.

%%%%%%%%%%%%%%%%%%%%%%%%%%%%%%%%%%%%%%%%%%%%%%%%%%%%%%%%%%%%%%%%%%%%%%

\setcounter{tocdepth}{2}
\tableofcontents

%%%%%%%%%%%%%%%%%%%%%%%%%%%%%%%%%%%%%%%%%%%%%%%%%%%%%%%%%%%%%%%%%%%%%%
%%%%%%%%%%%%%%%%%%%%%%%%%%%%%%%%%%%%%%%%%%%%%%%%%%%%%%%%%%%%%%%%%%%%%%
%%%%%%%%%%%%%%%%%%%%%%%%%%%%%%%%%%%%%%%%%%%%%%%%%%%%%%%%%%%%%%%%%%%%%%
%%%%%%%%%%%%%%%%%%%%%%%%%%%%%%%%%%%%%%%%%%%%%%%%%%%%%%%%%%%%%%%%%%%%%%
%%%%%%%%%%%%%%%%%%%%%%%%%%%%%%%%%%%%%%%%%%%%%%%%%%%%%%%%%%%%%%%%%%%%%%

\section{Nilpotent reductions}

Throughout this section, all representations are unitary representations
on complex Hilbert spaces. If \((V,\pi)\) is a representation of a group
\(G\) and \(H\leq G\), we write \(P_H\) for the orthogonal projection
onto \(V^H\).

\subsection{Finite-index reductions}

The cyclic escape property is stable under passage between a group and
its infinite finite-index subgroups. We shall use both directions.

\begin{lemma}
\label{lem:finite-index-overgroup-cyclic-escape}
Let \(G\) be an infinite countable group, and let \(H\leq G\) be an
infinite finite-index subgroup. If \(H\) has the cyclic escape property,
then \(G\) has the cyclic escape property.
\end{lemma}

\begin{proof}
Let \((V,\pi)\) be a totally ergodic unitary representation of \(G\).
The restricted representation of \(H\) is totally ergodic. Indeed, if
\(M\leq H\) has finite index, then \(M\) also has finite index in \(G\),
and therefore \(V^M=\{0\}\).

Let \(v\in V\) and let \(\varepsilon>0\). Since \(H\) has the cyclic
escape property, there exists an infinite cyclic subgroup \(C\leq H\)
such that
\[
        \|P_Cv\|<\varepsilon .
\]
Since \(C\) is also an infinite cyclic subgroup of \(G\), this proves
the claim.
\end{proof}

\begin{lemma}
\label{lem:finite-index-subgroup-cyclic-escape}
Let \(G\) be an infinite countable group, and let \(H\leq G\) be an
infinite finite-index subgroup. If \(G\) has the cyclic escape property,
then \(H\) has the cyclic escape property.
\end{lemma}

\begin{proof}
Let
\[
        K=\bigcap_{g\in G}gHg^{-1}
\]
be the normal core of \(H\). Then \(K\triangleleft G\), \(K\leq H\),
and \(K\) has finite index in \(G\). It is enough, by
Lemma~\ref{lem:finite-index-overgroup-cyclic-escape}, to prove that
\(K\) has the cyclic escape property.

Let \((V,\pi)\) be a totally ergodic unitary representation of \(K\).
Consider the induced representation
\[
        \Pi=\operatorname{Ind}_K^G\pi .
\]
Since \(K\) has finite index in \(G\), we may realize this representation
on
\[
        W=\bigoplus_{r\in R}V,
\]
where \(R\subset G\) is a set of representatives for \(G/K\).

We first check that \(\Pi\) is totally ergodic as a representation of
\(G\). Let \(M\leq G\) be finite index. The restriction of \(\Pi\) to
\(M\) decomposes over the finite \(M\)-orbits in \(G/K\). If \(rK\) is
one such coset, its stabilizer in \(M\) acts on the corresponding fibre
through the subgroup
\[
        r^{-1}Mr\cap K
\]
of \(K\). This subgroup has finite index in \(K\). Since \(\pi\) is
totally ergodic, it has no nonzero fixed vectors under
\(r^{-1}Mr\cap K\). Hence no \(M\)-orbit contributes \(M\)-fixed
vectors, and therefore
\[
        W^M=\{0\}.
\]
Thus \(\Pi\) is totally ergodic.

Now fix \(v\in V\), and let \(\xi\in W\) be the vector supported on the
coset \(K\) with value \(v\). Since \(G\) has the cyclic escape
property, for every \(\delta>0\) there exists an infinite cyclic
subgroup
\[
        C=\langle g\rangle\leq G
\]
such that
\[
        \|P_C\xi\|<\delta .
\]
Set
\[
        D=C\cap K .
\]
Since \(K\) has finite index in \(G\), the subgroup \(D\) has finite
index in \(C\), and hence is infinite cyclic.

Let \(m=[C:D]\). Equivalently, \(m\) is the length of the \(C\)-orbit of
the coset \(K\) in \(G/K\). On this orbit, the \(C\)-fixed vectors are
determined by their value at the coset \(K\), and this value must lie in
\(V^D\). Therefore the orthogonal projection of \(\xi\) onto the
\(C\)-fixed vectors supported on this orbit has norm
\[
        \frac{1}{\sqrt m}\|P_Dv\|.
\]
Since \(\xi\) is supported on this orbit, we get
\[
        \|P_C\xi\|^2=\frac{1}{m}\|P_Dv\|^2 .
\]
Thus
\[
        \|P_Dv\|
        =
        \sqrt m\,\|P_C\xi\|
        \leq
        \sqrt{[G:K]}\,\|P_C\xi\|.
\]
Choosing \(\delta\) arbitrarily small gives
\[
        \inf_D\|P_Dv\|=0,
\]
where \(D\) ranges over infinite cyclic subgroups of \(K\). Hence \(K\)
has the cyclic escape property. Therefore \(H\) has the cyclic escape
property.
\end{proof}

%%%%%%%%%%%%%%%%%%%%%%%%%%%%%%%%%%%%%%%%%%%%%%%%%%%%%%%%%%%%%%%%%%

\subsection{The lower central series}

We shall prove the cyclic escape property for finitely
generated nilpotent groups by induction on the nilpotency class. Recall
that the lower central series of a group $G$ is defined by
$\gamma_1(G)=G$ and $\gamma_{i+1}(G)=[G,\gamma_i(G)]$. If $G$ is
nilpotent of class $r$, then $\gamma_{r+1}(G)=\{e\}$ and
$\gamma_r(G)$ is a nontrivial central subgroup of $G$.

\begin{lemma}\label{lem:last-lower-central}
Let $G$ be a finitely generated nilpotent group of class $r\geq1$, and
put $Z=\gamma_r(G)$. Then $Z$ is a finitely generated abelian central
subgroup of $G$, and $G/Z$ is finitely generated nilpotent of class at
most $r-1$.
\end{lemma}

\begin{proof}
Since $\gamma_{r+1}(G)=\{e\}$, we have
$[G,\gamma_r(G)]=\{e\}$. Thus $Z=\gamma_r(G)$ is central in $G$, and in
particular abelian.

Subgroups of finitely generated nilpotent groups are finitely generated
by Hall's theorem; see \cite[Section~1C]{Segal_1983}. The quotient $G/Z$ is finitely generated,
and its lower central series is the image of the lower central series of
$G$. Since the image of $\gamma_r(G)$ is trivial in $G/Z$, the quotient
has nilpotency class at most $r-1$.
\end{proof}

Thus the induction step will be carried out for central extensions
$1\to Z\to G\to G/Z\to 1$, where $Z$ is finitely generated abelian and
central.

%%%%%%%%%%%%%%%%%%%%%%%%%%%%%%%%%%%%%%%%%%%%%%%%%%%%%%%%%%%%%%%%%%

\subsection{The central induction step}

The proof of Theorem~\ref{thm:main} will be reduced to
two representation-theoretic statements. The first is the abelian base
case.

\begin{proposition}\label{prop:abelian-base}
Every infinite finitely generated abelian group has the cyclic escape property.
\end{proposition}

The abelian base case is essentially the representation-theoretic form
of the directional expansion theorem for totally ergodic
$\mathbb Z^r$-actions proved by Björklund and Fish
\cite[Corollary 1.9]{Bjorklund_Fish_2024}. We include the short proof in the language of
unitary representations, since this is the form needed for the induction.

The second is the induction step over a central subgroup.

\begin{proposition}\label{prop:central-induction-step}
Let
\[
1\longrightarrow Z\longrightarrow G\xrightarrow{\,p\,}Q\longrightarrow 1
\]
be an exact sequence, where $Z$ is a finitely generated abelian central
subgroup of $G$. Suppose that $G$ is an infinite finitely generated
nilpotent group, that $Q$ is infinite, and that every infinite
finite-index subgroup of $Q$ has the cyclic escape property.
Then $G$ has the cyclic escape property.
\end{proposition}

The abelian base case will be proved in Section~6, and the central
induction step will be proved in Section~5. Once these two propositions
are established, Theorem~\ref{thm:main} follows by
induction on the nilpotency class; this final induction is carried out in
Subsection \ref{subsec:completion}.

%%%%%%%%%%%%%%%%%%%%%%%%%%%%%%%%%%%%%%%%%%%%%%%%%%%%%%%%%%%%%%%%%%

\subsection{Finite kernels}

\begin{lemma}
\label{lem:finite-kernel-transfer}
Let \(E\) be an infinite residually finite group, and let
\(F\triangleleft E\) be finite. Suppose that \(E/F\) has the cyclic
escape property. Then \(E\) has the cyclic escape property.
\end{lemma}

\begin{proof}
By residual finiteness, there is a finite-index normal subgroup
\(E_0\triangleleft E\) such that
\[
        E_0\cap F=\{e\}.
\]
Let \(q:E\to E/F\) be the quotient map. Then \(q|_{E_0}\) is injective,
and \(q(E_0)\) has finite index in \(E/F\). Since \(E_0\) is infinite,
so is \(q(E_0)\).

By Lemma~\ref{lem:finite-index-subgroup-cyclic-escape}, the group
\(q(E_0)\) has the cyclic escape property. Hence \(E_0\) has the cyclic
escape property. Since \(E_0\) has finite index in \(E\), Lemma
\ref{lem:finite-index-overgroup-cyclic-escape} implies that \(E\) has
the cyclic escape property.
\end{proof}

%%%%%%%%%%%%%%%%%%%%%%%%%%%%%%%%%%%%%%%%%%%%%%%%%%%%%%%%%%%%%%%%%%

\subsection{Finite-index quotients of the central subgroup}

Let
\[
1\longrightarrow Z\longrightarrow G\xrightarrow{\,p\,}Q\longrightarrow 1
\]
be an exact sequence, where $G$ is finitely generated nilpotent and
$Z$ is central and finitely generated abelian.

The following simple observation handles the case in which the quotient
is finite.

\begin{lemma}\label{lem:finite-quotient-central-case}
Assume Proposition~\ref{prop:abelian-base}. If $Q$ is finite and $G$ is
infinite, then $G$ has the cyclic escape property.
\end{lemma}

\begin{proof}
If $Q$ is finite, then $Z$ has finite index in $G$. Since $G$ is
infinite, $Z$ is infinite. By Proposition~\ref{prop:abelian-base}, the
group $Z$ has the cyclic escape property. Lemma~\ref{lem:finite-index-overgroup-cyclic-escape}
therefore implies that $G$ has the cyclic escape property.
\end{proof}

We shall also need to pass from \(G\) to quotients by finite-index
subgroups of \(Z\). If \(K\leq Z\) has finite index, then \(Z/K\) is a
finite normal subgroup of \(G/K\), and
\[
        (G/K)/(Z/K)\cong G/Z.
\]
Thus Lemma~\ref{lem:finite-kernel-transfer} allows us to remove the finite
central kernel \(Z/K\).

\begin{lemma}\label{lem:central-finite-index-quotient}
Assume that $Q$ is infinite, and suppose that every infinite
finite-index subgroup of $Q$ has the cyclic escape property.
Let $K\leq Z$ be finite index. Then $G/K$ has the cyclic escape property.
\end{lemma}

\begin{proof}
Put $E=G/K$ and $F=Z/K$. Then $F$ is a finite normal subgroup of $E$,
and $E/F$ is naturally isomorphic to $Q$. Since $Q$ is infinite, the
group $E$ is infinite. It is also finitely generated nilpotent.

Every infinite finite-index subgroup of $E/F$ has the cyclic escape property, because $E/F\cong Q$. Hence
Lemma~\ref{lem:finite-kernel-transfer} applies and gives the cyclic escape property for $E=G/K$.
\end{proof}

%%%%%%%%%%%%%%%%%%%%%%%%%%%%%%%%%%%%%%%%%%%%%%%%%%%%%%%%%%%%%%%%%%

\section{Central spectral decompositions}
\label{sec:central-spectral-decompositions}

Throughout this section, $G$ is a countable group, $Z\leq Z(G)$ is a
finitely generated abelian central subgroup, and $(V,\pi)$ is a unitary
representation of $G$. We write $\widehat Z$ for the compact dual group
of $Z$. Let $E_Z$ denote the projection-valued spectral measure for the
restriction of $\pi$ to $Z$. Thus, for $z\in Z$,
\[
\pi(z)=\int_{\widehat Z}\chi(z)\,dE_Z(\chi).
\]
If $A\subseteq \widehat Z$ is Borel, we write $V_A:=E_Z(A)V$. For more details about projection-valued spectral measures, see \cite[Chapter 2, Section 6]{Einsiedler_Ward_2025}.

\subsection{Central spectral subspaces}

The centrality of $Z$ implies that the spectral decomposition for $Z$ is
preserved by the whole group.

\begin{lemma}\label{lem:central-spectral-subspaces}
For every Borel set $A\subseteq\widehat Z$, the subspace $V_A$ is
$G$-invariant.
\end{lemma}

\begin{proof}
Let $g\in G$ and $z\in Z$. Since $Z$ is central, $\pi(g)$ commutes with
$\pi(z)$. Hence $\pi(g)$ commutes with the von Neumann algebra generated
by $\{\pi(z):z\in Z\}$, and therefore with each spectral projection
$E_Z(A)$. Thus $\pi(g)V_A=V_A$.
\end{proof}

For $v\in V$, let $\sigma_v$ be the spectral measure of $v$ for the
restriction to $Z$:
\[
\sigma_v(A)=\|E_Z(A)v\|^2
\]
for Borel $A\subseteq\widehat Z$. More generally, for $v,w\in V$, let
$\sigma_{v,w}$ be the complex spectral measure defined by
$\sigma_{v,w}(A)=\langle E_Z(A)v,w\rangle$.

We shall use the following consequence repeatedly.

\begin{lemma}\label{lem:central-spectral-type-preserved}
For every $g\in G$ and every $v\in V$,
\[
\sigma_{\pi(g)v}=\sigma_v.
\]
More generally, $\sigma_{\pi(g)v,\pi(g)w}=\sigma_{v,w}$ for all
$v,w\in V$.
\end{lemma}

\begin{proof}
Since $\pi(g)$ commutes with every spectral projection $E_Z(A)$,
\[
\sigma_{\pi(g)v}(A)
=
\|E_Z(A)\pi(g)v\|^2
=
\|\pi(g)E_Z(A)v\|^2
=
\sigma_v(A).
\]
The same commutation relation gives the corresponding identity for
cross-spectral measures.
\end{proof}

%%%%%%%%%%%%%%%%%%%%%%%%%%%%%%%%%%%%%%%%%%%%%%%%%%%%%%%%%%%%%%%%%%

\subsection{Periodic and aperiodic central spectrum}

A character $\chi\in\widehat Z$ is called \emph{periodic} if it has
finite order. We write $\Per(Z)$ for the set of periodic characters in
$\widehat Z$.

Since $Z$ is finitely generated abelian, a character
$\chi\in\widehat Z$ is periodic if and only if it is trivial on the subgroup $dZ$
for some $d\geq1$. Thus
\[
\Per(Z)=\bigcup_{d\geq1}\Ann(dZ),
\qquad
\Ann(dZ)=\{\chi\in\widehat Z:\chi(dz)=1\text{ for every }z\in Z\}.
\]
Since \(Z\) is finitely generated abelian, the quotient \(Z/dZ\) is
finite for every \(d\geq 1\). Hence
\(\operatorname{Ann}(dZ)\) is finite, being naturally identified with
the dual of \(Z/dZ\). The finiteness of these annihilators is what allows us below to replace
the periodic part of the central spectrum by finite spectral supports.

\begin{lemma}\label{lem:periodic-central-exhaustion}
Let $\rho$ be a finite positive Borel measure on $\widehat Z$. Then, for
every $\delta>0$, there exists $d\geq1$ such that
\[
\rho\bigl(\Per(Z)\setminus\Ann(dZ)\bigr)<\delta.
\]
\end{lemma}

\begin{proof}
For $k\geq1$, put $d_k=k!$. Then
\[
\Ann(d_1Z)\subseteq \Ann(d_2Z)\subseteq\cdots
\]
and
\[
\Per(Z)=\bigcup_{k\geq1}\Ann(d_kZ).
\]
Indeed, if $\chi\in\Per(Z)$ has order $r$, then
$\chi\in\Ann(d_kZ)$ whenever $r\mid d_k$. By $\sigma$-additivity,
\[
\rho(\Per(Z))
=
\lim_{k\to\infty}\rho(\Ann(d_kZ)).
\]
Therefore $\rho(\Per(Z)\setminus\Ann(d_kZ))<\delta$ for all sufficiently
large $k$.
\end{proof}

For $v\in V$, put
\[
v_{\mathrm{per}}:=E_Z(\Per(Z))v,
\qquad
v_{\mathrm{ap}}:=v-v_{\mathrm{per}}.
\]
Thus \(v_{\mathrm{per}}\) is supported on the periodic central spectrum,
whereas the spectral measure of \(v_{\mathrm{ap}}\) assigns measure zero
to \(\operatorname{Per}(Z)\). By
Lemma~\ref{lem:central-spectral-subspaces}, the subspaces
\[
        E_Z(\operatorname{Per}(Z))V
        \quad\text{and}\quad
        E_Z(\widehat Z\setminus \operatorname{Per}(Z))V
\]
are \(G\)-invariant, and \(v_{\mathrm{per}}\) and
\(v_{\mathrm{ap}}\) belong to these subspaces respectively.

Since $\Ann(dZ)$ is finite for every $d$,
Lemma~\ref{lem:periodic-central-exhaustion} implies that
\[
E_Z(\Ann(k!Z))v_{\mathrm{per}}\longrightarrow v_{\mathrm{per}}
\]
as $k\to\infty$.

%%%%%%%%%%%%%%%%%%%%%%%%%%%%%%%%%%%%%%%%%%%%%%%%%%%%%%%%%%%%%%%%%%

\subsection{Finite periodic supports}

Let $\Omega\subseteq\Per(Z)$ be finite, and put
$V_\Omega:=E_Z(\Omega)V$. By Lemma~\ref{lem:central-spectral-subspaces},
the subspace $V_\Omega$ is $G$-invariant.

Since each \(\chi\in\Omega\) has finite order, its image is finite.
Thus \(Z/\ker\chi\cong \chi(Z)\) is finite, so \(\ker\chi\) has finite
index in \(Z\). Since \(\Omega\) is finite, the subgroup
\[
        K_\Omega :=\bigcap_{\chi\in\Omega}\ker\chi
\]
has finite index in \(Z\).

\begin{lemma}\label{lem:finite-periodic-support-factors}
The subgroup $K_\Omega$ acts trivially on $V_\Omega$. Consequently, the
restricted representation of $G$ on $V_\Omega$ factors through the
quotient $G/K_\Omega$.
\end{lemma}

\begin{proof}
Let \(k\in K_\Omega\) and \(v\in V_\Omega\). Since
\(v=E_Z(\Omega)v\), we have
\[
        \pi(k)v
        =
        \int_\Omega \chi(k)\,dE_Z(\chi)v
        =
        \int_\Omega 1\,dE_Z(\chi)v
        =
        v.
\]
Thus \(K_\Omega\) acts trivially on \(V_\Omega\). Since \(K_\Omega\) is
central in \(G\), it is normal, and the restricted representation of
\(G\) on \(V_\Omega\) factors through \(G/K_\Omega\).
\end{proof}

When $G$ is finitely generated nilpotent and $Z$ is central, the quotient
$G/K_\Omega$ is again finitely generated nilpotent, and the image of
$Z$ in $G/K_\Omega$ is the finite central subgroup $Z/K_\Omega$. This is
the finite central kernel which will be removed using
Lemma~\ref{lem:finite-kernel-transfer}.

%%%%%%%%%%%%%%%%%%%%%%%%%%%%%%%%%%%%%%%%%%%%%%%%%%%%%%%%%%%%%%%%%%

\subsection{Approximation of the periodic part}

The preceding subsection applies to vectors supported on a finite subset
of $\Per(Z)$. We now record the corresponding approximation statement for
an arbitrary vector.

\begin{lemma}\label{lem:periodic-part-finite-approximation}
Let $v\in V$ and let $\eta>0$. Then there exists a finite set
$\Omega\subseteq\Per(Z)$ such that, with $u=E_Z(\Omega)v$, we have
\[
\|v_{\mathrm{per}}-u\|<\eta.
\]
Moreover, if $K_\Omega=\bigcap_{\chi\in\Omega}\ker\chi$, then
$K_\Omega$ has finite index in $Z$ and acts trivially on the
$G$-invariant subspace $V_\Omega=E_Z(\Omega)V$.
\end{lemma}

\begin{proof}
By Lemma~\ref{lem:periodic-central-exhaustion}, there exists $k\geq1$
such that
\[
\sigma_v\bigl(\Per(Z)\setminus \Ann(k!Z)\bigr)<\eta^2.
\]
Set $\Omega=\Ann(k!Z)$. Since $Z$ is finitely generated abelian,
$\Omega$ is finite. Also,
\[
\|v_{\mathrm{per}}-E_Z(\Omega)v\|^2
=
\sigma_v\bigl(\Per(Z)\setminus\Omega\bigr)
<
\eta^2.
\]
The remaining assertions follow from Lemma~\ref{lem:finite-periodic-support-factors}.
\end{proof}

Thus every vector may be decomposed, up to an arbitrarily small error,
as
\[
v=u+w+r,
\]
where $u$ is supported on a finite subset of $\Per(Z)$, the vector $w$
has spectral measure giving zero mass to $\Per(Z)$, and $\|r\|$ is as
small as desired. In the induction step, the finite periodic component
$u$ will be handled by passing to a quotient by $K_\Omega$, while the
aperiodic component $w$ will be handled by averaging over $K_\Omega$.

%%%%%%%%%%%%%%%%%%%%%%%%%%%%%%%%%%%%%%%%%%%%%%%%%%%%%%%%%%%%%%%%%%

\subsection{Cyclic projections on finite periodic supports}

Let $\Omega\subseteq\Per(Z)$ be finite, and let
$K_\Omega=\bigcap_{\chi\in\Omega}\ker\chi$. By
Lemma~\ref{lem:finite-periodic-support-factors}, the subspace
$V_\Omega=E_Z(\Omega)V$ is $G$-invariant and the representation of $G$
on $V_\Omega$ factors through $G/K_\Omega$.

\begin{lemma}\label{lem:periodic-support-quotient-projection}
Assume that \((V,\pi)\) is totally ergodic and that
\(G/K_\Omega\) has the cyclic escape property. Then, for every
\(u\in V_\Omega\) and every \(\varepsilon>0\), there exists an element
\(h\in G\) such that \(\langle hK_\Omega\rangle\) is an infinite cyclic
subgroup of \(G/K_\Omega\) and
\[
        \|P_{\langle h\rangle}u\|<\varepsilon .
\]
Moreover, for every \(z\in K_\Omega\),
\[
        \|P_{\langle zh\rangle}u\|=\|P_{\langle h\rangle}u\|.
\]
\end{lemma}

\begin{proof}
The quotient representation of $G/K_\Omega$ on $V_\Omega$ is totally
ergodic. Indeed, if $L/K_\Omega$ is a finite-index subgroup of
$G/K_\Omega$, then $L$ is finite index in $G$, and total ergodicity of
the $G$-representation gives $V^L=\{0\}$.

Since $G/K_\Omega$ has the cyclic escape property, there is an
infinite cyclic subgroup of $G/K_\Omega$ whose fixed-space projection of
$u$ has norm less than $\varepsilon$. Choose $h\in G$ so that
$\langle hK_\Omega\rangle$ is this cyclic subgroup. Since $K_\Omega$
acts trivially on $V_\Omega$, the projection of $u$ onto the
$\langle h\rangle$-fixed subspace agrees with its projection onto the
$\langle hK_\Omega\rangle$-fixed subspace in the quotient
representation. Hence $\|P_{\langle h\rangle}u\|<\varepsilon$.

Let \(z\in K_\Omega\). Since \(K_\Omega\) acts trivially on
\(V_\Omega\) and \(K_\Omega\) is central in \(G\), we have, for every
\(n\in\mathbb Z\) and every \(w\in V_\Omega\),
\[
        \pi((zh)^n)w=\pi(h^n)w.
\]
Indeed, \((zh)^n=z^n h^n\), and \(\pi(z^n)w=w\) on \(V_\Omega\).
Hence the \(\langle zh\rangle\)-fixed subspace inside \(V_\Omega\) is
the same as the \(\langle h\rangle\)-fixed subspace inside \(V_\Omega\).

Since \(V_\Omega\) is \(G\)-invariant, its orthogonal complement is also
\(G\)-invariant. Therefore the orthogonal projections
\(P_{\langle zh\rangle}\) and \(P_{\langle h\rangle}\) preserve
\(V_\Omega\). For \(u\in V_\Omega\), their restrictions to
\(V_\Omega\) are the projections onto the same subspace. Thus
\[
        P_{\langle zh\rangle}u=P_{\langle h\rangle}u,
\]
and hence
\[
        \|P_{\langle zh\rangle}u\|=\|P_{\langle h\rangle}u\|.
\]
\end{proof}

%%%%%%%%%%%%%%%%%%%%%%%%%%%%%%%%%%%%%%%%%%%%%%%%%%%%%%%%%%%%%%%%%%

\subsection{Restriction to finite-index subgroups of the centre}

Let $K\leq Z$ be a finite-index subgroup. We shall average over $K$ in
the next section. The following elementary observation is the reason why
aperiodicity for $Z$ remains visible after restriction to $K$.

\begin{lemma}\label{lem:finite-index-center-aperiodic}
Let $K\leq Z$ be finite index, and let $j\in\mathbb Z\setminus\{0\}$.
Then
\[
\{\chi\in\widehat Z:\chi(jk)=1\text{ for every }k\in K\}
\subseteq \Per(Z).
\]
\end{lemma}

\begin{proof}
Suppose that $\chi(jk)=1$ for every $k\in K$. Since $K$ has finite
index in $Z$, there is an integer $q\geq1$ such that $qz\in K$ for every
$z\in Z$. Hence $\chi(jqz)=1$ for every $z\in Z$. Thus $\chi$ has finite
order, so $\chi\in\Per(Z)$.
\end{proof}

\begin{corollary}\label{cor:aperiodic-cross-measures-on-K}
Let $u,w\in V$, and suppose that
$\sigma_u(\Per(Z))=0$ and $\sigma_w(\Per(Z))=0$. Then, for every
finite-index subgroup $K\leq Z$ and every $j\neq0$, the total variation
measure $|\sigma_{u,w}|$ gives zero mass to
\[
\{\chi\in\widehat Z:\chi(jk)=1\text{ for every }k\in K\}.
\]
\end{corollary}

\begin{proof}
Let \(A\subset \widehat Z\) be Borel. Since \(E_Z(A)\) is an
orthogonal projection,
\[
        |\sigma_{u,w}(A)|
        =
        |\langle E_Z(A)u,w\rangle|
        =
        |\langle E_Z(A)u,E_Z(A)w\rangle|
        \leq
        \|E_Z(A)u\|\,\|E_Z(A)w\|.
\]
Thus
\[
        |\sigma_{u,w}(A)|^2
        \leq
        \sigma_u(A)\sigma_w(A).
\]
In particular, if either \(\sigma_u(A)=0\) or \(\sigma_w(A)=0\), then
\(\sigma_{u,w}\) vanishes on every Borel subset of \(A\), and hence
\[
        |\sigma_{u,w}|(A)=0.
\]

Now take
\[
        A=
        \{\chi\in\widehat Z:\chi(jk)=1\text{ for every }k\in K\}.
\]
By Lemma~\ref{lem:finite-index-center-aperiodic}, this set is
contained in \(\operatorname{Per}(Z)\). Since
\(\sigma_u(\operatorname{Per}(Z))=0\) and
\(\sigma_w(\operatorname{Per}(Z))=0\), the preceding observation gives
\[
        |\sigma_{u,w}|(A)=0.
\]
\end{proof}

This corollary will be used with $w=\pi(g)u$. By
Lemma~\ref{lem:central-spectral-type-preserved}, if
$\sigma_u(\Per(Z))=0$, then also $\sigma_{\pi(g)u}(\Per(Z))=0$ for every
$g\in G$.

%%%%%%%%%%%%%%%%%%%%%%%%%%%%%%%%%%%%%%%%%%%%%%%%%%%%%%%%%%%%%%%%%%

\section{Aperiodic central averaging}

Throughout this section, $G$ is a countable group, $Z\leq Z(G)$ is a
finitely generated abelian central subgroup, and $(V,\pi)$ is a unitary
representation of $G$. We keep the spectral notation from
Section~\ref{sec:central-spectral-decompositions}: $E_Z$ is the spectral
measure for the restriction of $\pi$ to $Z$, and $\sigma_{u,w}$ denotes
the corresponding cross-spectral measure. If $L\leq G$, we write $P_L$
for the orthogonal projection onto $V^L$.

The goal of this section is to prove that central aperiodic spectrum can
be made invisible to a suitable cyclic projection after multiplying a
fixed group element by an element of a finite-index subgroup of $Z$.

\subsection{Averaging over finite-index subgroups of the centre}

Let $K\leq Z$ be finite index, and fix a Følner sequence $(F_R)$ in $K$.

\begin{lemma}\label{lem:central-fourier-averaging}
Let $K\leq Z$ be finite index, let $j\in\mathbb Z\setminus\{0\}$, and
let $u,w\in V$ satisfy
\[
\sigma_u(\Per(Z))=\sigma_w(\Per(Z))=0.
\]
Then
\[
\frac1{|F_R|}\sum_{z\in F_R}
\langle \pi(jz)u,w\rangle
\longrightarrow 0.
\]
\end{lemma}

\begin{proof}
For $z\in K$, we have
\[
\langle \pi(jz)u,w\rangle
=
\int_{\widehat Z}\chi(jz)\,d\sigma_{u,w}(\chi).
\]
Therefore
\[
\frac1{|F_R|}\sum_{z\in F_R}
\langle \pi(jz)u,w\rangle
=
\int_{\widehat Z}
\left(
\frac1{|F_R|}\sum_{z\in F_R}\chi(jz)
\right)
d\sigma_{u,w}(\chi).
\]

The functions inside the integral are bounded in absolute value by $1$.
For a fixed character $\chi$, the average
\[
\frac1{|F_R|}\sum_{z\in F_R}\chi(jz)
\]
converges to $0$ unless $\chi(jz)=1$ for every $z\in K$.

Let
\[
A=\{\chi\in\widehat Z:\chi(jz)=1\text{ for every }z\in K\}.
\]
By Lemma~\ref{lem:finite-index-center-aperiodic}, we have
$A\subseteq\Per(Z)$. By Corollary~\ref{cor:aperiodic-cross-measures-on-K},
the total variation measure $|\sigma_{u,w}|$ gives zero mass to $A$.
The integrands are
uniformly bounded by $1$ and converge pointwise to $0$ off the
$|\sigma_{u,w}|$-null set $A$. Hence dominated convergence, applied with
respect to $|\sigma_{u,w}|$, gives the result.
\end{proof}

%%%%%%%%%%%%%%%%%%%%%%%%%%%%%%%%%%%%%%%%%%%%%%%%%%%%%%%%%%%%%%%%%%

\subsection{The central averaging lemma}

We now prove the main averaging estimate for the aperiodic central
spectrum.

\begin{proposition}\label{prop:central-aperiodic-averaging}
Let $K\leq Z$ be finite index, let $h\in G$, and let $u\in V$ satisfy
\[
\sigma_u(\Per(Z))=0.
\]
Then
\[
\inf_{z\in K}\|P_{\langle zh\rangle}u\|=0.
\]
\end{proposition}

\begin{proof}
Fix a Følner sequence $(F_R)$ in $K$. For $M\geq1$ and $z\in K$, put
\[
A_M(z)u
:=
\frac1M\sum_{j=0}^{M-1}\pi((zh)^j)u .
\]
Since $Z$ is central, $(zh)^j=z^jh^j$. Thus, using additive notation
for $Z$,
\[
A_M(z)u
=
\frac1M\sum_{j=0}^{M-1}\pi(jz)\pi(h^j)u .
\]

For fixed $z$, the operator $\pi(zh)$ is unitary. By the mean ergodic
theorem applied to this unitary operator, the averages $A_M(z)u$ converge in norm
to $P_{\langle zh\rangle}u$. Moreover, if
$u=P_{\langle zh\rangle}u+u_0$ is the orthogonal decomposition of $u$
into its $\langle zh\rangle$-fixed part and its orthogonal complement,
then the orthogonal complement is invariant under $\pi(zh)$. Hence
\[
A_M(z)u
=
P_{\langle zh\rangle}u+A_M(z)u_0
\]
is an orthogonal decomposition. Therefore
\[
\|P_{\langle zh\rangle}u\|
\leq
\|A_M(z)u\|
\]
for every $M\geq1$.

Expanding the square and averaging over $F_R$, we get
\[
\frac1{|F_R|}\sum_{z\in F_R}\|A_M(z)u\|^2
=
\frac1{M^2}\sum_{i,j=0}^{M-1}
\frac1{|F_R|}\sum_{z\in F_R}
\left\langle
\pi((j-i)z)\pi(h^j)u,\pi(h^i)u
\right\rangle .
\]
If $i=j$, the inner term is $\|u\|^2$. If $i\neq j$, then
Lemma~\ref{lem:central-spectral-type-preserved} gives
\[
\sigma_{\pi(h^j)u}(\Per(Z))=
\sigma_{\pi(h^i)u}(\Per(Z))=0,
\]
and Lemma~\ref{lem:central-fourier-averaging} implies that the
corresponding inner average tends to $0$ as $R\to\infty$.

It follows that
\[
\limsup_{R\to\infty}
\frac1{|F_R|}\sum_{z\in F_R}\|A_M(z)u\|^2
\leq
\frac{\|u\|^2}{M}.
\]
Since $\|P_{\langle zh\rangle}u\|\leq\|A_M(z)u\|$ for every $z$ and
every $M$, we also have
\[
\limsup_{R\to\infty}
\frac1{|F_R|}\sum_{z\in F_R}\|P_{\langle zh\rangle}u\|^2
\leq
\frac{\|u\|^2}{M}
\]
for every $M\geq1$. Letting $M\to\infty$ gives
\[
\lim_{R\to\infty}
\frac1{|F_R|}\sum_{z\in F_R}\|P_{\langle zh\rangle}u\|^2=0.
\]
Therefore, for every $\varepsilon>0$, some $z\in K$ satisfies
$\|P_{\langle zh\rangle}u\|<\varepsilon$.
\end{proof}

%%%%%%%%%%%%%%%%%%%%%%%%%%%%%%%%%%%%%%%%%%%%%%%%%%%%%%%%%%%%%%%%%%

\subsection{Adding a finite periodic component}

The averaging lemma will be used after the periodic part of a vector has
already been controlled. We record the resulting estimate in a form that
will be applied in the central induction step.

\begin{corollary}\label{cor:periodic-aperiodic-combination}
Let $\Omega\subseteq\Per(Z)$ be finite, let
$K_\Omega=\bigcap_{\chi\in\Omega}\ker\chi$, and let $h\in G$. Suppose
that $u\in V_\Omega$ and $w\in V$ satisfy
\[
\sigma_w(\Per(Z))=0.
\]
Then, for every $\varepsilon>0$, there exists $z\in K_\Omega$ such that
\[
\|P_{\langle zh\rangle}(u+w)\|
\leq
\|P_{\langle h\rangle}u\|+\varepsilon.
\]
\end{corollary}

\begin{proof}
By Proposition~\ref{prop:central-aperiodic-averaging}, applied with
$K=K_\Omega$, there exists $z\in K_\Omega$ such that
\[
\|P_{\langle zh\rangle}w\|<\varepsilon.
\]

We claim that
\[
P_{\langle zh\rangle}u=P_{\langle h\rangle}u.
\]
Indeed, $K_\Omega$ acts trivially on $V_\Omega$, and $K_\Omega$ is
central in $G$. Hence, for every $n\in\mathbb Z$,
\[
\pi((zh)^n)u=\pi(h^n)u.
\]
Thus the $\langle zh\rangle$-fixed subspace and the $\langle h\rangle$-fixed
subspace agree inside $V_\Omega$. Since $V_\Omega$ is $G$-invariant, its
orthogonal complement is also $G$-invariant, so the orthogonal projections
onto these fixed spaces preserve $V_\Omega$. Therefore the two projections
agree on $u$.

It follows that
\[
\|P_{\langle zh\rangle}(u+w)\|
\leq
\|P_{\langle zh\rangle}u\|+\|P_{\langle zh\rangle}w\|
<
\|P_{\langle h\rangle}u\|+\varepsilon.
\]
\end{proof}

In particular, if $v=u+w+r$, where $u\in V_\Omega$,
$\sigma_w(\Per(Z))=0$, and $r\in V$, then the same choice of $z$ gives
\[
\|P_{\langle zh\rangle}v\|
\leq
\|P_{\langle h\rangle}u\|+\varepsilon+\|r\|.
\]

%%%%%%%%%%%%%%%%%%%%%%%%%%%%%%%%%%%%%%%%%%%%%%%%%%%%%%%%%%%%%%%%%%

\section{The central induction step}

Throughout this section, let
\[
1\longrightarrow Z\longrightarrow G\xrightarrow{\,p\,}Q\longrightarrow 1
\]
be an exact sequence, where $G$ is an infinite finitely generated
nilpotent group and $Z\leq Z(G)$ is a finitely generated abelian central
subgroup. Let $(V,\pi)$ be a totally ergodic unitary representation of
$G$. We keep the notation of Sections~3 and~4 for the spectral
decomposition over $\widehat Z$.

The standing assumption in this section is that every infinite
finite-index subgroup of $Q$ has the cyclic escape property.
Our goal is to prove that $G$ has the cyclic escape property.

\subsection{Controlling the periodic approximation}

We first show that a finite periodic central spectral component can be
controlled by the induction hypothesis.

\begin{lemma}\label{lem:periodic-approximation-controlled}
Let $\Omega\subseteq\Per(Z)$ be finite, and let $u\in V_\Omega$. Then,
for every $\varepsilon>0$, there exists $h\in G$ such that
\[
\|P_{\langle h\rangle}u\|<\varepsilon
\]
and $\langle hK_\Omega\rangle$ is an infinite cyclic subgroup of
$G/K_\Omega$, where $K_\Omega=\bigcap_{\chi\in\Omega}\ker\chi$.
\end{lemma}

\begin{proof}
By Lemma~\ref{lem:central-finite-index-quotient}, the quotient
\(G/K_\Omega\) has the cyclic escape property. The
representation of \(G\) on \(V_\Omega\) factors through \(G/K_\Omega\)
by Lemma~\ref{lem:finite-periodic-support-factors}.

We claim that the resulting representation of \(G/K_\Omega\) on
\(V_\Omega\) is totally ergodic. Let
\(M/K_\Omega\leq G/K_\Omega\) be finite index. Then \(M\leq G\) has
finite index. Since the original representation of \(G\) is totally
ergodic,
\[
        V^M=\{0\}.
\]
The fixed space of \(M/K_\Omega\) in \(V_\Omega\) is
\[
        (V_\Omega)^{M/K_\Omega}
        =
        \{u\in V_\Omega:\pi(m)u=u\text{ for every }m\in M\}
        =
        V_\Omega\cap V^M
        =
        \{0\}.
\]
Thus the quotient representation is totally ergodic.

Applying the cyclic escape property of \(G/K_\Omega\) to
\(u\in V_\Omega\), we find an infinite cyclic subgroup of
\(G/K_\Omega\) whose fixed projection of \(u\) has norm less than
\(\varepsilon\). Choose \(h\in G\) whose image generates this cyclic
subgroup. Since \(K_\Omega\) acts trivially on \(V_\Omega\), the fixed
projection for \(\langle h\rangle\) on \(u\) agrees with the fixed
projection for \(\langle hK_\Omega\rangle\) in the quotient
representation. Hence
\[
        \|P_{\langle h\rangle}u\|<\varepsilon.
\]
\end{proof}

%%%%%%%%%%%%%%%%%%%%%%%%%%%%%%%%%%%%%%%%%%%%%%%%%%%%%%%%%%%%%%%%%%

\subsection{Proof of the central induction step}

We now prove Proposition~\ref{prop:central-induction-step}.

\begin{proof}[Proof of Proposition~\ref{prop:central-induction-step}]
Let $(V,\pi)$ be a totally ergodic unitary representation of $G$, and
let $v\in V$. We must show that, for every $\varepsilon>0$, there is an
infinite cyclic subgroup $C\leq G$ such that $\|P_Cv\|<\varepsilon$.

Choose $\eta>0$ so small that $4\eta<\varepsilon$. By
Lemma~\ref{lem:periodic-part-finite-approximation}, there exists a
finite set $\Omega\subseteq\Per(Z)$ such that, with
$u=E_Z(\Omega)v$, we have
\[
\|v_{\mathrm{per}}-u\|<\eta.
\]
Put $w=v_{\mathrm{ap}}$ and $r=v_{\mathrm{per}}-u$. Then
\[
v=u+w+r,
\qquad
\sigma_w(\Per(Z))=0,
\qquad
\|r\|<\eta.
\]

By Lemma~\ref{lem:periodic-approximation-controlled}, there exists
$h\in G$ such that
\[
\|P_{\langle h\rangle}u\|<\eta
\]
and $\langle hK_\Omega\rangle$ is an infinite cyclic subgroup of
$G/K_\Omega$.

By Corollary~\ref{cor:periodic-aperiodic-combination}, applied to
$u,w$ and this $h$, there exists $z\in K_\Omega$ such that
\[
\|P_{\langle zh\rangle}(u+w)\|
<
\|P_{\langle h\rangle}u\|+\eta
<
2\eta.
\]
Since projections are contractions,
\[
\|P_{\langle zh\rangle}v\|
\leq
\|P_{\langle zh\rangle}(u+w)\|+\|r\|
<
3\eta
<
\varepsilon.
\]

It remains only to check that the subgroup $\langle zh\rangle$ is
infinite cyclic. Since $z\in K_\Omega$, the elements $zh$ and $h$ have
the same image in $G/K_\Omega$. By construction,
$\langle hK_\Omega\rangle$ is infinite cyclic. Hence the image of
$zh$ in $G/K_\Omega$ has infinite order, and therefore $zh$ itself has
infinite order in $G$. Thus $\langle zh\rangle$ is an infinite cyclic
subgroup.

Thus $G$ has the cyclic escape property.
\end{proof}

%%%%%%%%%%%%%%%%%%%%%%%%%%%%%%%%%%%%%%%%%%%%%%%%%%%%%%%%%%%%%%%%%%

\section{The abelian base case and completion}

In this section we prove the abelian base case and then complete the
proof of Theorem~\ref{thm:main}. The abelian case is the
representation-theoretic form of the directional expansion theorem for
totally ergodic actions of finitely generated abelian groups; see
Björklund and Fish~\cite{Bjorklund_Fish_2024} for a more general version. We include the proof in the
unitary representation language used in this paper.

\subsection{The abelian base case}

We first record a simple averaging lemma for finitely generated free
abelian groups.

\begin{lemma}\label{lem:abelian-direction-selection}
Let $A\cong\mathbb Z^r$ with $r\geq1$, and let $\rho$ be a finite
positive Borel measure on $\widehat A$ such that $\rho(\Per(A))=0$.
Then, for every $\delta>0$, there exists a nonzero $a\in A$ such that
\[
\rho(\Ann(a))<\delta.
\]
\end{lemma}

\begin{proof}
Choose a Følner sequence $(F_R)$ in $A$ given by boxes in some basis of
$A$. We claim that, for every $\chi\in\widehat A\setminus\Per(A)$,
\[
\frac{1}{|F_R|}\bigl|\{a\in F_R:\chi(a)=1\}\bigr|\longrightarrow0.
\]
Indeed, if $\chi$ is not periodic, then its kernel has infinite index in
$A$. Hence $\ker\chi$ has rank strictly smaller than $r$, and therefore
has zero density in the boxes $F_R$.

By dominated convergence,
\[
\frac1{|F_R|}\sum_{a\in F_R}\rho(\Ann(a))
=
\int_{\widehat A}
\frac{1}{|F_R|}\bigl|\{a\in F_R:\chi(a)=1\}\bigr|
\,d\rho(\chi)
\longrightarrow0.
\]
Since $\Ann(0)=\widehat A$, the contribution of $a=0$ to the average is
at most $\rho(\widehat A)/|F_R|$, which tends to $0$. Therefore
\[
\frac1{|F_R\setminus\{0\}|}
\sum_{a\in F_R\setminus\{0\}}\rho(\Ann(a))
\longrightarrow0.
\]
For all sufficiently large $R$, this average is less than $\delta$.
Hence some nonzero $a\in F_R$ satisfies $\rho(\Ann(a))<\delta$.
\end{proof}

\begin{proof}[Proof of Proposition~\ref{prop:abelian-base}]
Let $G$ be an infinite finitely generated abelian group. By the
structure theorem for finitely generated abelian groups, $G$ contains a
finite-index subgroup $A\cong\mathbb Z^r$ with $r\geq1$. By
Lemma~\ref{lem:finite-index-overgroup-cyclic-escape}, it is enough to prove that $A$
has the cyclic escape property. We may therefore assume that
$G\cong\mathbb Z^r$.

Let $(V,\pi)$ be a totally ergodic unitary representation of $G$, and
let $v\in V$. Let $\sigma_v$ be the spectral measure of $v$ on
$\widehat G$. We first note that $\sigma_v(\Per(G))=0$. Indeed,
\[
\Per(G)=\bigcup_{d\geq1}\Ann(dG),
\]
and each subgroup $dG$ has finite index in $G$. Since the representation
is totally ergodic, $V^{dG}=\{0\}$ for every $d\geq1$. The projection
onto $V^{dG}$ is the spectral projection $E_G(\Ann(dG))$, so
$\sigma_v(\Ann(dG))=0$ for every $d$.

Let $\varepsilon>0$. By Lemma~\ref{lem:abelian-direction-selection},
there exists a nonzero $g\in G$ such that
\[
\sigma_v(\Ann(g))<\varepsilon^2.
\]
Since $G$ is free abelian, $\langle g\rangle$ is infinite cyclic. The
projection onto $V^{\langle g\rangle}$ is the spectral projection onto
$\Ann(g)$. Therefore
\[
\|P_{\langle g\rangle}v\|^2
=
\sigma_v(\Ann(g))
<
\varepsilon^2.
\]
Thus $\|P_{\langle g\rangle}v\|<\varepsilon$, proving the directional
projection property.
\end{proof}

%%%%%%%%%%%%%%%%%%%%%%%%%%%%%%%%%%%%%%%%%%%%%%%%%%%%%%%%%%%%%%%%%%

\subsection{Completion of the nilpotent theorem}
\label{subsec:completion}

We now prove Theorem~\ref{thm:main}.

\begin{proof}[Proof of Theorem~\ref{thm:main}]
We prove that \(G\) has the cyclic escape property by induction on the
nilpotency class of \(G\).

If \(G\) has nilpotency class one, then \(G\) is finitely generated
abelian. Since \(G\) is infinite, Proposition~\ref{prop:abelian-base}
gives the cyclic escape property for \(G\).

Assume now that the theorem has been proved for all infinite finitely
generated nilpotent groups of nilpotency class at most \(r-1\), and let
\(G\) be an infinite finitely generated nilpotent group of class
\(r\geq 2\). Put
\[
        Z=\gamma_r(G).
\]
By Lemma~\ref{lem:last-lower-central}, \(Z\) is a finitely generated
abelian central subgroup of \(G\), and
\[
        Q=G/Z
\]
is finitely generated nilpotent of class at most \(r-1\).

If \(Q\) is finite, then Lemma~\ref{lem:finite-quotient-central-case}
gives the cyclic escape property for \(G\).

Assume that \(Q\) is infinite. Let \(M\leq Q\) be an infinite
finite-index subgroup. Then \(M\) is again a finitely generated
nilpotent group of nilpotency class at most \(r-1\). By the induction
hypothesis, \(M\) has the cyclic escape property. Therefore every
infinite finite-index subgroup of \(Q\) has the cyclic escape property.

The hypotheses of Proposition~\ref{prop:central-induction-step} are now
satisfied for the central extension
\[
        1\longrightarrow Z\longrightarrow G\longrightarrow Q
        \longrightarrow 1 .
\]
Hence \(G\) has the cyclic escape property.

This proves that every infinite finitely generated nilpotent group has
the cyclic escape property. The final assertion, that every totally
ergodic probability-measure-preserving action of \(G\) is directionally
expansive, follows from
Lemma~\ref{lem:cyclic-escape-implies-directional-expansion}.
\end{proof}

%%%%%%%%%%%%%%%%%%%%%%%%%%%%%%%%%%%%%%%%%%%%%%%%%%%%%%%%%%%%%%%%%%

\section{Higher-rank lattices}
\label{sec:higher-rank-lattices}

In this section we prove Theorem~\ref{thm:higher-rank-lattices} and
Corollary~\ref{cor:SL-PSL-PGLnZ}. The argument is independent of the
nilpotent argument above, but it has the same Hilbert-space target: one
must find cyclic subgroups whose fixed-vector projections are small.

The proof is based on stationary character rigidity for higher-rank
lattices, due to Boutonnet--Houdayer. We first prove the weakly mixing
case. Under the additional finite-image hypothesis for
finite-dimensional unitary representations, total ergodicity implies
weak mixing, giving the cyclic escape property.

Throughout the main part of the section, \(L\) denotes a connected
simple Lie group with trivial centre, no compact factors, and real rank
at least two, and \(G<L\) denotes a lattice. The finite-centre reduction needed for part (2) of
Theorem~\ref{thm:higher-rank-lattices} is carried out at the end of the
section.

\subsection{\texorpdfstring{\(C^*\)-algebras, states, and traces}{C*-algebras, states, and traces}}
\label{subsec:cstar-states-traces}

Let \((V,\pi)\) be a unitary representation of a countable group \(G\).
We denote by
\[
        C^*_\pi(G)\subset B(V)
\]
the norm-closed unital \(*\)-subalgebra generated by \(\pi(G)\). Thus
\(C^*_\pi(G)\) is the closure, in the operator norm on \(B(V)\), of the
complex linear span of \(\{\pi(g):g\in G\}\).

A \emph{state} on a unital \(C^*\)-algebra \(A\) is a linear functional
\(\omega:A\to \mathbb C\) such that \(\omega(1)=1\) and
\[
        \omega(a^*a)\geq 0
        \qquad(a\in A).
\]
The state space \(S(A)\) is weak-* compact by the Banach--Alaoglu
theorem. If \(\xi\in V\) is a unit vector, then
\[
        \omega_\xi(a)=\langle a\xi,\xi\rangle
        \qquad(a\in C^*_\pi(G))
\]
is a state, called the vector state associated to \(\xi\).

A state \(\omega\) on \(A\) is \emph{tracial} if
\[
        \omega(ab)=\omega(ba)
        \qquad(a,b\in A).
\]

We also use the left regular representation
\[
        \lambda_G:G\to \mathcal U(\ell^2(G)),
        \qquad
        \lambda_G(s)\delta_h=\delta_{sh}.
\]
The associated reduced group \(C^*\)-algebra is
\[
        C^*_{\lambda_G}(G)\subset B(\ell^2(G)).
\]
It has a canonical tracial state, the regular trace,
\[
        \tau_{\lambda_G}(a)=\langle a\delta_e,\delta_e\rangle.
\]
In particular,
\[
        \tau_{\lambda_G}(\lambda_G(s))=
        \begin{cases}
        1,& s=e,\\
        0,& s\neq e.
        \end{cases}
\]

The group \(G\) acts on \(C^*_\pi(G)\) by conjugation:
\[
        \alpha_g(a)=\pi(g)a\pi(g)^{-1}
        \qquad(a\in C^*_\pi(G)).
\]
Equivalently,
\[
        \alpha_g(\pi(s))=\pi(gsg^{-1})
        \qquad(g,s\in G).
\]
This induces an affine action of \(G\) on \(S(C^*_\pi(G))\) by
\[
        g\omega=\omega\circ\alpha_{g^{-1}}.
\]

%%%%%%%%%%%%%%%%%%%%%%%%%%%%%%%%%%%%%%%%%%%%%%%%%%%%%%%%%%%%%%%%%%

\subsection{Furstenberg measures and stationary characters}
\label{subsec:furstenberg-stationary-characters}

Let \(P<L\) be a minimal parabolic subgroup. The homogeneous space
\(L/P\) is the Furstenberg boundary of \(L\). Fix a maximal compact
subgroup \(K<L\), and let \(\nu\) be the \(K\)-invariant probability
measure on \(L/P\).

A probability measure \(\kappa\) on \(G\) is called a \emph{Furstenberg
measure} if its support generates \(G\) as a semigroup and the Poisson
boundary of the random walk \((G,\kappa)\) is \((L/P,\nu)\). Such
measures exist by Furstenberg's boundary theory; see
\cite[Chapter~VI]{Margulis_1991}. We fix such a measure \(\kappa\) for
the rest of the section.

Let \(G\) act continuously and affinely on a compact convex set
\(\mathcal K\). A point \(x\in\mathcal K\) is called
\emph{\(\kappa\)-stationary} if
\[
        x=\sum_{g\in G}\kappa(g)\,gx.
\]
Equivalently, for every continuous affine function \(F\) on \(\mathcal K\),
\[
        F(x)=\sum_{g\in G}\kappa(g)F(gx).
\]

\begin{lemma}
\label{lem:stationary-point-exists}
Let \(\mathcal K\) be a nonempty weak-* compact convex \(G\)-invariant subset of
the state space \(S(A)\) of a unital \(C^*\)-algebra \(A\), where \(G\)
acts on \(S(A)\) by affine weak-* homeomorphisms. Then \( \mathcal K\) contains a
\(\kappa\)-stationary point.
\end{lemma}

\begin{proof}
Define an affine map \(T:\mathcal K\to \mathcal K\) by
\[
        T\omega=\sum_{g\in G}\kappa(g)\,g\omega .
\]
The sum is understood in the weak-* topology. The map \(T\) is weak-*
continuous, since for every \(a\in A\),
\[
        (T\omega)(a)=\sum_{g\in G}\kappa(g)(g\omega)(a),
\]
and the series converges uniformly on \(S(A)\) after truncating the
\(\kappa\)-tail. Also \(T(\mathcal K)\subset \mathcal K\), because \(\mathcal K\) is weak-* closed
and convex.

Fix \(\omega_0\in \mathcal K\), and set
\[
        \omega_N=\frac1N\sum_{n=0}^{N-1}T^n\omega_0 .
\]
Let \(\omega\) be a weak-* cluster point of \((\omega_N)\). Since
\[
        T\omega_N-\omega_N
        =
        \frac1N(T^N\omega_0-\omega_0),
\]
we have \(T\omega_N-\omega_N\to 0\) weak-*. Hence \(T\omega=\omega\).
Thus \(\omega\) is \(\kappa\)-stationary.
\end{proof}

A normalized positive-definite function is a function
\(\chi:G\to\mathbb C\) such that \(\chi(e)=1\) and
\[
        \sum_{i,j=1}^n \overline{c_i}c_j\chi(g_i^{-1}g_j)\geq 0
\]
for all \(n\geq 1\), \(c_1,\ldots,c_n\in\mathbb C\), and
\(g_1,\ldots,g_n\in G\). A normalized positive-definite function
\(\chi\) is called a \emph{\(\kappa\)-character} if
\[
        \chi(s)=\sum_{g\in G}\kappa(g)\chi(g^{-1}sg)
        \qquad(s\in G).
\]
Every conjugation-invariant normalized positive-definite function is a
\(\kappa\)-character.

If \(\omega\) is a \(\kappa\)-stationary state on \(C^*_\pi(G)\), then
\[
        \chi_\omega(s)=\omega(\pi(s))
        \qquad(s\in G)
\]
is a \(\kappa\)-character. Indeed, \(\chi_\omega\) is normalized and
positive definite, and stationarity gives
\[
\begin{aligned}
        \chi_\omega(s)
        &=
        \sum_{g\in G}\kappa(g)(g\omega)(\pi(s))  \\
        &=
        \sum_{g\in G}\kappa(g)\omega(\alpha_{g^{-1}}(\pi(s))) \\
        &=
        \sum_{g\in G}\kappa(g)\omega(\pi(g^{-1}sg)) \\
        &=
        \sum_{g\in G}\kappa(g)\chi_\omega(g^{-1}sg).
\end{aligned}
\]

%%%%%%%%%%%%%%%%%%%%%%%%%%%%%%%%%%%%%%%%%%%%%%%%%%%%%%%%%%%%%%%%%%

\subsection{The Boutonnet--Houdayer input}
\label{subsec:BH-input}

We now record the rigidity theorem used in the conjugacy averaging
argument.

\begin{theorem}[Boutonnet--Houdayer]
\label{thm:BH-input}
Let \(L\) be a connected simple Lie group with trivial centre, no
compact factors, and real rank at least two. Let \(G<L\) be a lattice,
and let \(\kappa\) be the Furstenberg measure fixed above.

\begin{enumerate}
\item
Every \(\kappa\)-character of \(G\) is conjugation invariant.

\item
Let \((V,\pi)\) be a weakly mixing unitary representation of \(G\).
Then \(C^*_\pi(G)\) has a unique tracial state \(\tau_\pi\). This trace
is determined on the canonical generators by
\[
        \tau_\pi(\pi(g))=
        \begin{cases}
        1, & g=e,\\
        0, & g\neq e .
        \end{cases}
\]
\end{enumerate}
\end{theorem}

The assertion that the assignment \(\pi(g)\mapsto\lambda_G(g)\) is
well defined is part of the theorem. In the present setting, one can also
see why no nontrivial element is collapsed by \(\pi\). Indeed, if
\(N=\ker\pi\), then \(N\triangleleft G\). By Margulis' normal subgroup
theorem, \(N\) is either finite or finite index. The finite-index case
would force \(\pi\) to factor through a finite group, contradicting weak
mixing. The finite case is trivial: a finite normal subgroup of a lattice
in a centre-free semisimple Lie group is contained in the centre, by
Borel density. Since \(L\) has trivial centre, \(\ker\pi=\{e\}\).

\begin{proof}[Reference]
The first statement is \cite[Theorem~A]{Boutonnet_Houdayer_2021}. For the second statement, \cite[Corollary~D]{Boutonnet_Houdayer_2021}
implies that, for a weakly mixing representation \(\pi\), every tracial
state on \(C^*_\pi(G)\) factors through the reduced group
\(C^*\)-algebra, after excluding the finite-dimensional character terms.
Those terms are excluded by weak mixing; see
\cite[Lemma~6.5]{Boutonnet_Houdayer_2021}.
\end{proof}

%%%%%%%%%%%%%%%%%%%%%%%%%%%%%%%%%%%%%%%%%%%%%%%%%%%%%%%%%%%%%%%%%%

\subsection{Conjugacy averaging}
\label{subsec:conjugacy-averaging}

\begin{lemma}
\label{lem:higher-rank-conjugacy-averaging}
Let \(s_1,\ldots,s_m\in G\setminus\{e\}\), and let \((V,\pi)\) be a
weakly mixing unitary representation of \(G\). Then, for every
\(v\in V\),
\[
        0\in
        \overline{\operatorname{conv}}
        \left\{
        \bigl(
        \langle \pi(gs_1g^{-1})v,v\rangle,\ldots,
        \langle \pi(gs_mg^{-1})v,v\rangle
        \bigr):g\in G
        \right\}
        \subset \mathbb C^m .
\]
\end{lemma}

\begin{proof}
The case \(v=0\) is immediate. Assume that \(v\neq 0\), and put
\[
        \xi=\frac{v}{\|v\|}.
\]
Suppose that the conclusion fails. Viewing \(\mathbb C^m\) as a real
locally convex vector space, the Hahn--Banach separation theorem gives
\(c_1,\ldots,c_m\in\mathbb C\) and \(\delta>0\) such that
\[
        \operatorname{Re}
        \sum_{j=1}^m c_j
        \langle \pi(gs_jg^{-1})\xi,\xi\rangle
        \geq \delta
        \qquad(g\in G).
\]
We use the separation theorem in the form stated in
\cite[Theorem~3.4]{Rudin_1991}.

Let
\[
        A=C^*_\pi(G),
\]
and let \(\omega_\xi\in S(A)\) be the vector state associated to
\(\xi\). Let \(K\) be the weak-* closed convex hull of the orbit
\[
        \{g\omega_\xi:g\in G\}\subset S(A).
\]
Then \(K\) is nonempty, weak-* compact, convex, and \(G\)-invariant. By
Lemma~\ref{lem:stationary-point-exists}, \(K\) contains a
\(\kappa\)-stationary state \(\omega\).

Define
\[
        \chi(s)=\omega(\pi(s))
        \qquad(s\in G).
\]
By the preceding subsection, \(\chi\) is a \(\kappa\)-character. Hence,
by Theorem~\ref{thm:BH-input}, \(\chi\) is conjugation invariant.

We claim that \(\omega\) is tracial on
\(C^*_\pi(G)\). Indeed, if \(x,y\in G\), then \(xy\) and \(yx\) are
conjugate in \(G\), since
\[
        yx=y(xy)y^{-1}.
\]
Hence
\[
\begin{aligned}
        \omega(\pi(x)\pi(y))
        &=
        \omega(\pi(xy))        \\
        &=
        \chi(xy)               \\
        &=
        \chi(yx)               \\
        &=
        \omega(\pi(yx))        \\
        &=
        \omega(\pi(y)\pi(x)).
\end{aligned}
\]
By linearity, \(\omega(ab)=\omega(ba)\) for all \(a,b\) in the group
algebra \(\mathbb C[G]\), viewed inside \(C^*_\pi(G)\). Since
\(\mathbb C[G]\) is norm dense in \(C^*_\pi(G)\) and \(\omega\) is norm
continuous, \(\omega\) is tracial on \(C^*_\pi(G)\).

Since \(\pi\) is weakly mixing, Theorem~\ref{thm:BH-input}, part (2),
implies that \(\omega\) is the unique tracial state on \(A=C^*_\pi(G)\).
Hence, for every \(s\in G\),
\[
        \omega(\pi(s))
        =
        \begin{cases}
        1, & s=e,\\
        0, & s\neq e .
        \end{cases}
\]
In particular,
\[
        \omega(\pi(s_j))=0
        \qquad(j=1,\ldots,m),
\]
because \(s_j\neq e\).

Now define
\[
        F:S(A)\to\mathbb R,
        \qquad
        F(\omega')
        =
        \operatorname{Re}
        \sum_{j=1}^m c_j\omega'(\pi(s_j)).
\]
The map \(F\) is weak-* continuous and affine. For every \(g\in G\),
\[
\begin{aligned}
        F(g^{-1}\omega_\xi)
        &=
        \operatorname{Re}
        \sum_{j=1}^m c_j
        (g^{-1}\omega_\xi)(\pi(s_j))  \\
        &=
        \operatorname{Re}
        \sum_{j=1}^m c_j
        \omega_\xi(\alpha_g(\pi(s_j))) \\
        &=
        \operatorname{Re}
        \sum_{j=1}^m c_j
        \langle \pi(gs_jg^{-1})\xi,\xi\rangle \\
        &\geq \delta.
\end{aligned}
\]
Since \(F\) is affine and weak-* continuous, the same inequality holds
on \(K\). In particular,
\[
        F(\omega)\geq \delta.
\]
On the other hand, \(\omega(\pi(s_j))=0\) for all \(j\), so
\(F(\omega)=0\), a contradiction.
\end{proof}

%%%%%%%%%%%%%%%%%%%%%%%%%%%%%%%%%%%%%%%%%%%%%%%%%%%%%%%%%%%%%%%%%%

\subsection{Conjugates of cyclic subgroups}
\label{subsec:conjugate-cyclic-subgroups}

For a subgroup \(C\leq G\), let \(P_C\) denote the \emph{orthogonal
projection} onto the \(C\)-fixed subspace.

\begin{proposition}
\label{prop:higher-rank-small-cyclic-projections}
Let \(t\in G\) have infinite order, and let \((V,\pi)\) be a weakly
mixing unitary representation of \(G\). Then, for every \(v\in V\),
\[
        \inf_{g\in G}\|P_{\langle gtg^{-1}\rangle}v\|=0.
\]
\end{proposition}

\begin{proof}
Fix \(v\in V\) and \(\eta>0\). For \(N\geq 1\) and \(g\in G\), set
\[
        M_N(g)=\frac{1}{N}\sum_{r=0}^{N-1}\pi(gt^rg^{-1}),
        \qquad
        A_N(g)=M_N(g)^*M_N(g).
\]
Then \(A_N(g)\) is a positive self-adjoint bounded operator. Moreover,
\[
\begin{aligned}
        \langle A_N(g)v,v\rangle
        &=
        \frac{1}{N^2}
        \sum_{r,q=0}^{N-1}
        \langle \pi(gt^{q-r}g^{-1})v,v\rangle  \\
        &=
        \frac{\|v\|^2}{N}
        +
        \sum_{0<|k|<N}
        \frac{N-|k|}{N^2}
        \langle \pi(gt^kg^{-1})v,v\rangle .
\end{aligned}
\]

Choose \(N\) so large that
\[
        \frac{\|v\|^2}{N}<\frac{\eta}{2}.
\]
Since \(t\) has infinite order, \(t^k\neq e\) for all \(0<|k|<N\).
By Lemma~\ref{lem:higher-rank-conjugacy-averaging}, there are
\(g_1,\ldots,g_\ell\in G\) and \(a_1,\ldots,a_\ell\geq 0\), with
\[
        \sum_{i=1}^{\ell}a_i=1,
\]
such that
\[
        \left|
        \sum_{i=1}^{\ell} a_i
        \langle \pi(g_it^kg_i^{-1})v,v\rangle
        \right|
        <
        \frac{\eta}{2}
        \qquad(0<|k|<N).
\]
Averaging the preceding identity over \(i\), taking real parts, and
using absolute values gives
\[
\begin{aligned}
        \sum_{i=1}^{\ell}a_i
        \langle A_N(g_i)v,v\rangle
        &\leq
        \frac{\|v\|^2}{N}
        +
        \sum_{0<|k|<N}
        \frac{N-|k|}{N^2}
        \left|
        \sum_{i=1}^{\ell}a_i
        \langle \pi(g_it^kg_i^{-1})v,v\rangle
        \right| .
\end{aligned}
\]
Since
\[
        \sum_{0<|k|<N}\frac{N-|k|}{N^2}<1,
\]
we obtain
\[
\begin{aligned}
        \sum_{i=1}^{\ell}a_i
        \langle A_N(g_i)v,v\rangle
        &<
        \frac{\eta}{2}
        +
        \frac{\eta}{2}
        =
        \eta .
\end{aligned}
\]
Each term \(\langle A_N(g_i)v,v\rangle\) is nonnegative. Hence, for
some \(i\),
\[
        \langle A_N(g_i)v,v\rangle<\eta .
\]

Let
\[
        C_i=\langle g_itg_i^{-1}\rangle
\]
and put
\[
        U_i=\pi(g_itg_i^{-1}).
\]
Define
\[
        \phi_N(z)
        =
        \left|
        \frac1N\sum_{r=0}^{N-1}z^r
        \right|^2,
        \qquad z\in\mathbb T .
\]
By functional calculus,
\[
        A_N(g_i)=\phi_N(U_i).
\]
Moreover,
\[
        \phi_N(z)\geq 0
        \quad\text{for all }z\in\mathbb T,
        \qquad
        \phi_N(1)=1.
\]
Hence
\[
        1_{\{1\}}\leq \phi_N
\]
pointwise on \(\mathbb T\).

Let \(E_i\) be the spectral measure of \(U_i\). The \(C_i\)-fixed
subspace is the \(1\)-eigenspace of \(U_i\), so
\[
        P_{C_i}=E_i(\{1\})=1_{\{1\}}(U_i).
\]
Therefore
\[
        P_{C_i}
        =
        1_{\{1\}}(U_i)
        \leq
        \phi_N(U_i)
        =
        A_N(g_i)
\]
in the positive-operator order. Consequently,
\[
        \|P_{C_i}v\|^2
        =
        \langle P_{C_i}v,v\rangle
        \leq
        \langle A_N(g_i)v,v\rangle
        <
        \eta .
\]
Since \(\eta>0\) was arbitrary, this proves
\[
        \inf_{g\in G}\|P_{\langle gtg^{-1}\rangle}v\|=0.
\]
\end{proof}

\begin{proof}[Proof of Theorem~\ref{thm:higher-rank-lattices}, part (1)]
Let
\[
        G\curvearrowright (X,\mu)
\]
be a weakly mixing probability-measure-preserving action, and let
\(\pi\) be the Koopman representation on \(L^2_0(X,\mu)\). Then
\(\pi\) is weakly mixing.

We first choose an element of infinite order in \(G\). Since \(L\) has
trivial centre, the adjoint representation embeds \(L\) into
\(\operatorname{GL}(\mathfrak l)\), where \(\mathfrak l\) is the Lie
algebra of \(L\). Thus \(G\) is linear. Moreover, \(G\) is finitely
generated, since lattices in connected semisimple Lie groups are
finitely generated; see \cite[Remark~13.21]{Raghunathan_1972}. If every
element of \(G\) had finite order, then Schur's theorem would imply that
\(G\) is finite; see \cite{Schur_1911}. This is impossible, since \(L\)
is noncompact and \(G\) is a lattice. Hence there exists
\[
        t\in G
\]
of infinite order.

By Proposition~\ref{prop:higher-rank-small-cyclic-projections}, for
every \(f\in L^2_0(X,\mu)\),
\[
        \inf_{g\in G}
        \|P_{\langle gtg^{-1}\rangle}f\|=0 .
\]
We now apply the conditional-expectation estimate from
Lemma~\ref{lem:cyclic-escape-implies-directional-expansion}. Let
\(B\subset X\) be measurable with \(\mu(B)>0\), and put
\[
        f=1_B-\mu(B).
\]
For every \(\varepsilon>0\), the preceding display gives some
\(g\in G\) such that
\[
        \|P_{\langle gtg^{-1}\rangle}f\|_2
\]
is sufficiently small. The estimate in
Lemma~\ref{lem:cyclic-escape-implies-directional-expansion} then gives
\[
        \mu(\langle gtg^{-1}\rangle B)>1-\varepsilon .
\]
Thus the action is directionally expansive.
\end{proof}

%%%%%%%%%%%%%%%%%%%%%%%%%%%%%%%%%%%%%%%%%%%%%%%%%%%%%%%%%%%%%%%%%%

\subsection{From total ergodicity to weak mixing}
\label{subsec:total-ergodicity-weak-mixing}

\begin{lemma}
\label{lem:total-ergodic-implies-weak-mixing}
Let \(G\) be an infinite countable group such that every
finite-dimensional unitary representation of \(G\) has finite image.
Then every totally ergodic unitary representation of \(G\) is weakly
mixing.
\end{lemma}

\begin{proof}
Let \((V,\pi)\) be a totally ergodic unitary representation of \(G\).
Suppose that \(\pi\) is not weakly mixing. Then \(\pi\) has a nonzero
finite-dimensional invariant subspace \(W\subset V\). Let
\[
        \rho:G\to\mathcal U(W)
\]
be the corresponding finite-dimensional representation. By assumption,
\(\rho(G)\) is finite. Hence
\[
        H=\ker\rho
\]
has finite index in \(G\).

The subgroup \(H\) acts trivially on \(W\). Thus every vector in \(W\)
is \(H\)-fixed. Since \((V,\pi)\) is totally ergodic, we have
\[
        V^H=\{0\}.
\]
This contradicts \(W\neq\{0\}\). Hence \(\pi\) is weakly mixing.
\end{proof}

The finite-image hypothesis is used exactly to pass from a
finite-dimensional obstruction to weak mixing to a finite-index subgroup
acting trivially on that obstruction. It is not a formal consequence of
total ergodicity.

\begin{proof}[Proof of Theorem~\ref{thm:higher-rank-lattices}, part (2), when \(L\) has trivial centre]
Let \((V,\pi)\) be a totally ergodic unitary representation of \(G\),
and let \(v\in V\). By
Lemma~\ref{lem:total-ergodic-implies-weak-mixing}, the representation
\((V,\pi)\) is weakly mixing.

Since \(G\) is infinite and linear, it contains an element \(t\in G\) of
infinite order, as in the proof of part (1). Therefore
Proposition~\ref{prop:higher-rank-small-cyclic-projections} gives
\[
        \inf_{g\in G}
        \|P_{\langle gtg^{-1}\rangle}v\|=0.
\]
Since the subgroups \(\langle gtg^{-1}\rangle\) are infinite cyclic,
this proves that \(G\) has the cyclic escape property.

The action-level assertion follows from
Lemma~\ref{lem:cyclic-escape-implies-directional-expansion}.
\end{proof}

%%%%%%%%%%%%%%%%%%%%%%%%%%%%%%%%%%%%%%%%%%%%%%%%%%%%%%%%%%%%%%%%%%

\subsection{Arithmetic examples}
\label{subsec:arithmetic-examples}

We now record a general criterion for the \emph{finite-image hypothesis}
appearing in Theorem~\ref{thm:higher-rank-lattices}. We use standard algebraic-group notation. If \(\mathbf L\) is a real
algebraic group, then \(\mathbf L(\mathbb R)\) denotes its group of real
points, and \(\mathbf L(\mathbb R)^\circ\) denotes the identity
component in the usual real topology. The adjective \emph{adjoint} means
that the algebraic group has trivial centre. A representation of
\(\mathbf L(\mathbb R)^\circ\) is called algebraic if it is the
restriction to \(\mathbf L(\mathbb R)^\circ\) of a rational
representation of \(\mathbf L\).

\begin{lemma}
\label{lem:unitary-representations-finite-image-general}
Let \(\mathbf L\) be a connected adjoint simple algebraic group over
\(\mathbb R\), and let \(L=\mathbf L(\mathbb R)^\circ\). Assume that
\(L\) is generated by its unipotent one-parameter subgroups. Let
\(\Gamma<L\) be a lattice with the following superrigidity property:
whenever
\[
        \rho:\Gamma\to \operatorname{GL}(W)
\]
is a finite-dimensional representation with infinite image, there is a
finite-index subgroup \(\Gamma_0\leq \Gamma\) such that
\(\rho|_{\Gamma_0}\) agrees with the restriction of an algebraic
representation
\[
        \overline{\rho}:L\to \operatorname{GL}(W).
\]
Then every finite-dimensional unitary representation of \(\Gamma\) has
finite image.
\end{lemma}

\begin{proof}
Let
\[
        \rho:\Gamma\to \mathcal U(W)
\]
be a finite-dimensional unitary representation. Suppose that \(\rho\)
has infinite image. By the assumed superrigidity property, after
replacing \(\Gamma\) by a finite-index subgroup \(\Gamma_0\), the
representation \(\rho|_{\Gamma_0}\) agrees with the restriction of an
algebraic representation
\[
        \overline{\rho}:L\to \operatorname{GL}(W).
\]

By the Borel density theorem, \(\Gamma_0\) is Zariski dense in
\(\mathbf L\). Since \(\rho(\Gamma_0)\) preserves the Hermitian form on
\(W\), and this preservation condition is real algebraic,
\(\overline{\rho}(L)\) preserves the same Hermitian form. Hence
\[
        \overline{\rho}(L)\subset \mathcal U(W).
\]

Let \(u\in L\) be unipotent. Since \(\overline{\rho}\) is algebraic,
\(\overline{\rho}(u)\) is unipotent. It is also unitary. A unitary
unipotent operator is the identity. Hence \(\overline{\rho}\) is trivial
on every unipotent one-parameter subgroup of \(L\). Since these
subgroups generate \(L\), the representation \(\overline{\rho}\) is
trivial.

Thus \(\rho\) is trivial on the finite-index subgroup \(\Gamma_0\).
Therefore \(\rho(\Gamma)\) is finite, contradicting the assumption that
it was infinite.
\end{proof}

The hypotheses of
Lemma~\ref{lem:unitary-representations-finite-image-general} hold for
the standard higher-rank arithmetic lattices considered here. The
required virtual algebraic extension is the compact-target form of
Margulis superrigidity; see \cite[Chapter~VII, Theorems~5.1
and~7.1(b)]{Margulis_1991} and
\cite[Corollary~16.4.1]{Morris_2015}. For the groups below, the ambient
real group is generated by its unipotent one-parameter subgroups.

\begin{corollary}
\label{cor:PSLnZ-finite-image}
For every \(n\geq 3\), every finite-dimensional unitary representation
of \(\operatorname{PSL}_n(\mathbb Z)\) has finite image.
\end{corollary}

\begin{proof}
Apply Lemma~\ref{lem:unitary-representations-finite-image-general} to
\[
        \Gamma=\operatorname{PSL}_n(\mathbb Z)
        <
        L=\operatorname{PSL}_n(\mathbb R).
\]
The group \(\operatorname{PSL}_n(\mathbb R)\) is generated by its
elementary unipotent one-parameter subgroups.
\end{proof}

More generally, the same argument applies to many standard higher-rank
arithmetic lattices. For instance, after passing to the adjoint group
and the identity component of the real points, it applies to
\[
        \operatorname{PSp}_{2n}(\mathbb Z),
        \qquad n\geq 2,
\]
and to arithmetic orthogonal lattices associated to integral quadratic
forms of signature \((p,q)\) with
\[
        \min(p,q)\geq 2
        \qquad\text{and}\qquad
        p+q\geq 5 .
\]
Equivalently, for such a form \(f\), it applies to lattices commensurable
with
\[
        \mathbf{PO}_f(\mathbb Z)\cap \mathbf{PO}_f(\mathbb R)^\circ .
\]
For the construction and classification of these standard arithmetic
lattices in classical groups, see
\cite[Proposition~15.15 and Figure~15.2]{Morris_2015}; for the
superrigidity input used above, see
\cite[Chapter~VII, Theorems~5.1 and~7.1(b)]{Margulis_1991} or
\cite[Corollary~16.4.1]{Morris_2015}. Thus these groups, and their
finite-index variants, have the same finite-image property for
finite-dimensional unitary representations.

\begin{proof}[Proof of Corollary~\ref{cor:SL-PSL-PGLnZ}]
First consider
\[
        \operatorname{PSL}_n(\mathbb Z)
        <
        \operatorname{PSL}_n(\mathbb R).
\]
By Corollary~\ref{cor:PSLnZ-finite-image}, every finite-dimensional
unitary representation of \(\operatorname{PSL}_n(\mathbb Z)\) has finite
image. Hence Theorem~\ref{thm:higher-rank-lattices}, part (2), implies
that \(\operatorname{PSL}_n(\mathbb Z)\) has the cyclic escape property.

The natural quotient map
\[
        \operatorname{SL}_n(\mathbb Z)
        \longrightarrow
        \operatorname{PSL}_n(\mathbb Z)
\]
has finite central kernel. The group \(\operatorname{SL}_n(\mathbb Z)\)
is residually finite, for instance by reduction modulo primes. Therefore
Lemma~\ref{lem:finite-kernel-transfer} implies that
\(\operatorname{SL}_n(\mathbb Z)\) has the cyclic escape property.

Finally,
\[
        \operatorname{PSL}_n(\mathbb Z)
        \leq
        \operatorname{PGL}_n(\mathbb Z)
\]
has finite index. Since \(\operatorname{PSL}_n(\mathbb Z)\) has the
cyclic escape property,
Lemma~\ref{lem:finite-index-overgroup-cyclic-escape} implies that
\(\operatorname{PGL}_n(\mathbb Z)\) has the cyclic escape property.

The action-level conclusion for all three groups follows from
Lemma~\ref{lem:cyclic-escape-implies-directional-expansion}.
\end{proof}

The finite-image hypothesis in Theorem~\ref{thm:higher-rank-lattices},
part (2), is not automatic for higher-rank lattices. For example, there
are irreducible arithmetic lattices
\[
        \Gamma<\operatorname{SL}_3(\mathbb R)\times \operatorname{SU}(3);
\]
see \cite[Proposition~15.20]{Morris_2015}. Let $G$
be the projection of such a lattice to \(\operatorname{SL}_3(\mathbb R)\).
Since \(\operatorname{SU}(3)\) is compact, \(G\) is a lattice in
\(\operatorname{SL}_3(\mathbb R)\), after passing to a finite-index
subgroup if necessary to remove the finite kernel of the projection.
Since \(\Gamma\) is irreducible, its projection to \(\operatorname{SU}(3)\)
is dense. Hence \(G\) admits a finite-dimensional unitary representation
with dense image in \(\operatorname{SU}(3)\).

Therefore \(G\) admits a totally ergodic compact action which is not
weakly mixing and not directionally expansive, by
Theorem~\ref{thm:compact-obstruction}. This shows that the finite-image
hypothesis in Theorem~\ref{thm:higher-rank-lattices}, part (2), cannot
be omitted from the argument.

%%%%%%%%%%%%%%%%%%%%%%%%%%%%%%%%%%%%%%%%%%%%%%%%%%%%%%%%%%%%%%%%%%%%%%

\subsection{The finite-centre reduction}
\label{subsec:finite-centre-reduction}

\begin{proof}[Completion of the proof of Theorem~\ref{thm:higher-rank-lattices}]
It remains only to prove part (2) when \(L\) has finite centre. Let
\[
        \overline L=L/Z(L)
\]
be the adjoint quotient, and let
\[
        q:L\to \overline L
\]
be the quotient homomorphism. Then \(\overline L\) is a connected
noncompact simple Lie group with trivial centre and real rank at least
two. The image
\[
        \overline G=q(G)
\]
is a lattice in \(\overline L\), and the kernel of
\[
        q|_G:G\to \overline G
\]
is finite.

We first check that \(\overline G\) satisfies the same finite-image
hypothesis. Let
\[
        \rho:\overline G\to \mathcal U(W)
\]
be a finite-dimensional unitary representation. Then
\[
        \rho\circ q|_G:G\to \mathcal U(W)
\]
has finite image by assumption. Since \(q|_G\) maps \(G\) onto
\(\overline G\), it follows that \(\rho(\overline G)\) is finite.

By the trivial-centre case of part (2), the group \(\overline G\) has
the cyclic escape property. Since \(G\) is a lattice in a connected
semisimple Lie group with finite centre, it is finitely generated by
\cite[Remark~13.21]{Raghunathan_1972}. It is also linear, and hence
residually finite by Mal'cev's theorem~\cite{Malcev_1940}. Therefore
Lemma~\ref{lem:finite-kernel-transfer} applies to
\[
        1\longrightarrow \ker(q|_G)
        \longrightarrow G
        \longrightarrow \overline G
        \longrightarrow 1
\]
and shows that \(G\) has the cyclic escape property. This proves part
(2) in the finite-centre case.
\end{proof}

%%%%%%%%%%%%%%%%%%%%%%%%%%%%%%%%%%%%%%%%%%%%%%%%%%%%%%%%%%%%%%%%%%%%%%%%
%%%%%%%%%%%%%%%%%%%%%%%%%%%%%%%%%%%%%%%%%%%%%%%%%%%%%%%%%%%%%%%%%%%%%%%%
%%%%%%%%%%%%%%%%%%%%%%%%%%%%%%%%%%%%%%%%%%%%%%%%%%%%%%%%%%%%%%%%%%%%%%%%
%%%%%%%%%%%%%%%%%%%%%%%%%%%%%%%%%%%%%%%%%%%%%%%%%%%%%%%%%%%%%%%%%%%%%%%%
%%%%%%%%%%%%%%%%%%%%%%%%%%%%%%%%%%%%%%%%%%%%%%%%%%%%%%%%%%%%%%%%%%%%%%%%

\end{document}